\documentclass[11pt,reqno,twoside,makeidx]{amsart}
\usepackage[margin=2.7cm]{geometry}

\usepackage[english]{babel}
\usepackage{csquotes}

\usepackage{caption}
\usepackage{graphicx}
\usepackage{dsfont}
\usepackage{amsthm}
\usepackage{amsmath}
\usepackage{amsfonts}
\usepackage{amscd}
\usepackage{fancyhdr}
\usepackage{enumerate}
\usepackage{enumitem}
\usepackage{url}
\usepackage{lipsum}
\usepackage{hyperref} 

\usepackage{tikz}
\usepackage{tikz-cd}
    \usetikzlibrary{cd}
    \usetikzlibrary{arrows}
    \usetikzlibrary{matrix}
    \usepackage{pictexwd,dcpic}
    \usepackage{sidecap}
    \graphicspath{./imagine/}

\usepackage{caption}
\usepackage{subcaption}

\usepackage{cancel}
\usepackage{import}

\usepackage{amssymb}

\usepackage{mathtools}
\usepackage{layout}

\usepackage{booktabs}
\usepackage{array}
\usepackage{seqsplit}
\newcolumntype{L}[1]{>{\raggedright\arraybackslash}p{#1}}

\newcommand{\fund}[1]{ \pi_1 \left( #1 \right)}
\newcommand{\compl}[1]{E \left( #1 \right)}

\newcommand\restr[2]{{
  \left.\kern-\nulldelimiterspace 
  #1 
  \vphantom{|} 
  \right|_{#2} 
  }}

 \newcommand{\Ui}{U(1)}
 \newcommand{\normalsubgroup}[1]{\langle\! \langle #1 \rangle \! \rangle}
\newcommand{\R}[1]{\mathcal{R}_{U(1)} \left( #1\right)}
\newcommand{\geo}[2]{\Delta \left( #1 ,#2\right)}

\tikzset{cross/.style={cross out, draw=black, minimum size=2*(#1-\pgflinewidth), inner sep=0pt, outer sep=0pt},
cross/.default={1pt}}

\DeclareMathSymbol{\shortminus}{\mathbin}{AMSa}{"39}

\usetikzlibrary{decorations.pathmorphing}

\usepackage{mathrsfs}

\usepackage{multicol}
\usepackage{lscape}
\usepackage{cjhebrew}

\usepackage{bbold}

\usepackage{graphicx,nicefrac}
 \usepackage{framed}
\newcommand{\nequiv}{\not\equiv}

\usepackage{algorithm}
\usepackage{algpseudocode}
\algtext*{EndFor}
\algtext*{EndIf}
\usepackage{tikz}
\makeatletter
\newtheorem*{rep@teo}{\rep@title}
\newcommand{\newreptheorem}[2]{%
\newenvironment{rep#1}[1]{%
 \def\rep@title{#2 \ref{##1}}%
 \begin{rep@teo}}%
 {\end{rep@teo}}}
 \newreptheorem{teo}{Theorem}
 \newreptheorem{cor}{Corollary}
 \newreptheorem{prop}{Proposition}
 \newreptheorem{defn}{Definition}
 \newreptheorem{lemma}{Lemma}
 \newreptheorem{condition}{Condition}

\theoremstyle{definition}
\newtheorem{defn}{Definition}[section]

\newtheorem{exmp}[defn]{Example}
\newtheorem{rmk}[defn]{Remark}
\newtheorem{notation}[defn]{Notation}

\newtheorem*{PN*}{Please Note}

\theoremstyle{plain}
\newtheorem{teo}[defn]{Theorem}
\newtheorem*{teo*}{Theorem}
\newtheorem{cor}[defn]{Corollary}
\newtheorem{lemma}[defn]{Lemma}
\newtheorem{prop}[defn]{Proposition}

\newtheorem*{prob*}{Problem}

\newtheorem*{qst*}{Question}

\numberwithin{equation}{section}
\makeatother

\usepackage[backend=bibtex, style=alphabetic]{biblatex}
\usepackage{overpic}

\title[Some knots with no Su(2)-Abelian Surgeries]{Some knots with no Su(2)-Abelian Surgeries}
\author{Giacomo Bascape}

\begin{document}

\begin{abstract}
A knot in $S^3$ is said to be \emph{not} $SU(2)$-abelian knot, if every non-trivial
surgery along it yields a $3$-manifold whose fundamental group
admits an irreducible $SU(2)$-representation.
We provide examples of knots in $S^3$ that are not $SU(2)$-abelian.
A knot $K$ is said to be $SU(2)$-\emph{clean} if
whenever the fundamental group of the integer $r$-surgery $\fund{K(r)}$ has 
no $SU(2)$-irreducible representation,
then the Alexander polynomial of $K$ does not vanish at any $r$-th root of unity.
We show that if $K$ is a non-trivial $SU(2)$-clean knot,
the connected sum $K\#K$ is never $SU(2)$-abelian.
Finally, combining known results on $SU(2)$-abundant knots
and a classical epimorphism between knot groups,
we show that \emph{every} non-torus knot with at most $9$ crossings is not $SU(2)$-abelian,
with at most the two exceptions of $9_{47}$ and $9_{49}$.
\end{abstract}

\maketitle

\section{Introduction}
In this work, we consider only compact and oriented $3$-manifolds. 
For a given property $\mathcal{P}$ of closed $3$-manifolds, we say that a knot in $S^3$ is a $\mathcal{P}$-knot if there exists a non-trivial surgery on it yielding a $3$-manifold with property $\mathcal{P}$. 
For instance, the right-handed trefoil $3_1$ is an \emph{L-space knot}, since the $1$-surgery on it yields a Heegaard Floer L-space.
For the definition of Heegaard Floer homology, we refer the reader to \cite{HFdef2} and \cite{HFdef1}. 
A natural question in knot theory is to characterize which knots in $S^3$ are $\mathcal{P}$-knots. 

Among the various properties that a $3$-manifold may possess, those defined in terms of $SU(2)$-representations of the fundamental group have proven to be particularly deep and powerful. 
The study of $SU(2)$-representations plays a central role in low-dimensional topology, providing access to subtle geometric and topological features. 
A classical result such as the proof of Property P \cite{KronheimerMrowkaDehnSurgeryFundamental}
highlights the strength of $SU(2)$-representation techniques in distinguishing $3$-manifolds and detecting essential topological phenomena. 
In this perspective, understanding when a $3$-manifold is $SU(2)$-abelian, meaning that its fundamental group admits only abelian $SU(2)$-representations, provides a natural direction of inquiry.
Therefore, in this work, we take $\mathcal{P}$ to be the property of being $SU(2)$-abelian and, in particular, we shall study knots that are \emph{not} $SU(2)$-abelian;
said differently, we will study knots in $S^3$ such that every non-trivial surgery yields a $3$-manifold that is not $SU(2)$-abelian.

Torus knots admit infinitely many surgeries with cyclic fundamental group,
and therefore they are examples of $SU(2)$-abelian knots. 
Conversely, iterated torus knots $C_{p_1,q_1}C_{p_2,q_2}(T_{p_3,q_3})$ with $p_i \ge 2$ are examples of knots that are not $SU(2)$-abelian by \cite[Theorem 4.1]{Sivek_Zentner}.
As a general fact, it is worth mentioning that if a knot $K$ is not $SU(2)$-abelian, then it is \emph{abundant} in the sense of \cite[Definition 1.1]{SU2AbundantLargeFormula}.

We expect composite knots in $S^3$, namely those obtained as the connected sum of two non-trivial knots, to be non-$SU(2)$-abelian knots. Indeed, it is known that composite knots are not L-space knots \cite[Corollary 3.11]{CompositeKnotNLS}.
Moreover, Zhang's conjecture \cite[Conjecture~4]{zhang2019remarks} predicts that a non-L-space knot should be non-$SU(2)$-abelian.
In light of these results, one naturally expects that composite knots provide examples of knots that are not $SU(2)$-abelian.
This work provides a partial answer to this question.

To obtain a concrete result, we introduce an additional hypothesis. Namely, we consider $SU(2)$\emph{-clean} knots and prove
that for a non-trivial knot $K \subset S^3$, the connected sum $K^2=K\#K$ is not an $SU(2)$-abelian knot.
A knot $K$ is called $SU(2)$-clean if, whenever the integer $r$-surgery
on $K$ is $SU(2)$-abelian, the Alexander polynomial of $K$ does not vanish at any $r$-th root of unity.

\begin{teo}\label{teo: double knot}
    Let $K \subset S^3$ be a non-trivial $SU(2)$-clean knot. Then $K^2=K\#K$ is not an $SU(2)$-abelian knot.
\end{teo}

The $SU(2)$-clean condition appears naturally in the theory.
For instance, if $K$ is a Berge knot and $r \in \mathbb{Z}$ is a cyclic surgery slope
for $K$, then the Alexander polynomial of $K$
does not vanish at any $r$-th root of unity (see Theorem \ref{thm:berge-clean}).
More generally, every cyclic surgery is $SU(2)$-abelian, while cyclic surgery slopes are
necessarily integral by the Cyclic Surgery Theorem of Culler, Gordon, Luecke, and Shalen.
In this sense, the $SU(2)$-clean condition may be viewed as a natural extension of a phenomenon already present
in the classical theory of cyclic surgeries.
At present, the author is not aware of any knot that is not $SU(2)$-clean.
On the contrary, as shown in Corollary \ref{cor: some examples}, torus knots, $2$-bridge knots, and the particularly important pretzel knot $P(-2,3,7)$ are $SU(2)$-clean.
Moreover, any knot that admits no non-trivial $SU(2)$-abelian surgery is automatically $SU(2)$-clean.

Let $K$ be a knot in $S^3$, we denote by $E(K)$ the manifold $S^3 \setminus \nu(K)$, where $\nu(K)$ is an open tubular neighborhood of the knot $K$.
For a given knot $K \subset S^3$, the space $T(K,\partial)$ is defined as the set of representations $\fund{\partial \compl{K}} \to \Ui$ that extend to $\fund{\compl{K}}$.
Here $\Ui$ is the subgroup of diagonal matrices of $SU(2)$. Details can be found in \cite{Mino}.
In particular, $T(K,\partial)$ makes it possible to study $SU(2)$-representations
via geometric methods, by looking at subsets of the torus $\mbox{Hom}(\fund{\partial \compl{K}},U(1))$,
without resorting to gauge theory or instanton Floer homology.

Given two knots $K_1$ and $K_2$ in $S^3$, we define $K_{\#} \subset S^3$ as the composite knot $K_1\#K_2$.
For $i \in \{1,2,\#\}$, let $\{\mu_i,\lambda_i\}$ be an ordered basis of meridian and null-homologous longitude of $K_i$.
The fundamental group of $\compl{K_\#}$ admits a relatively easy presentation:
\begin{align}
\label{eq: presentazione composite knot}
    \fund{\compl{K_\#}} = \frac{\fund{\compl{K_1}}\ast \fund{\compl{K_2}}}{\normalsubgroup{\mu_1=\mu_2=\mu_\# , \lambda_\#=\lambda_1 \lambda_2}}.
\end{align}
The presentation \eqref{eq: presentazione composite knot} can be used to describe explicitly 
the invariant $T(K_1 \# K_2,\partial)$ as shown in Lemma \ref{lemma: descriviamo composite knots}.
Theorem \ref{teo: double knot} is proven within the study of $T(K\#K,\partial)$.

Furthermore, the invariant $T(K,\partial)$ is exhaustively described for torus knots in \cite[Section I.B]{EpKlassenSu2} and \cite{Mino}. Therefore, if $K_1$ and $K_2$ are both torus knots, the invariant $T(K_1 \# K_2,\partial)$ is completely determined by Lemma \ref{lemma: descriviamo composite knots}. As a consequence, we can prove the following proposition:

\begin{cor}\label{cor: composite torus knots are not SUa}
    Let $T_{p_1,p_2}$ and $T_{p_3,p_4}$ be two non-trivial torus knots, then $T_{p_1,p_2} \# T_{p_3,p_4}$ is not an $SU(2)$-abelian knot.
\end{cor}

The final part of this paper is devoted to the problem of detecting non-$SU(2)$-abelian knots among knots
with small crossing number. A quick search in KnotInfo \cite{knotinfo}
shows that every L-space knot with at most nine crossings is a torus knot.
Since torus knots are classical examples of $SU(2)$-abelian knots,
this suggests that the remaining knots with at most $9$ crossings should fail to be $SU(2)$-abelian.
By combining the results on $SU(2)$-abundant knots in \cite{SU2AbundantLargeFormula} with the existence of epimorphisms established in \cite{ApartialOrderInTheKnotTableII},
we obtain the following result, leaving only two exceptions.
\begin{cor}\label{cor: knots up to 9 crosssings}
    Let $K$ be a knot with at most $9$ crossings. If $K$ is neither a torus knot nor in $\{9_{47},9_{49}\}$,
    then
    $K$ is not $SU(2)$-abelian.
\end{cor}
Table \ref{tab:su2-master}, at the end of this work,
summarizes the $SU(2)$-abelian status of every knot with
at most $9$ crossings, together with the argument used
to establish it.

\subsection*{Acknowledgment}
I thank my former advisors, Steven Boyer and Duncan McCoy,
for teaching me the craft of mathematical research;
and my partner, Ju Marlow, for his unwavering support.

\section{Notation}\label{sec: notation}
In this section, we fix the notation and conventions that will be used throughout this work. 
We begin by recalling the standard notation for satellite knots and their associated spaces. 
Let $P(J)$ denote the satellite knot with pattern $P$ and companion knot $J$. 
The winding number of $P$ is defined as the algebraic intersection number between $P$ and an essential disc in $V = S^1 \times \mathbb{D}^2$. 
The exterior of the satellite $P(J)$ is obtained by gluing $V_P$ to $\compl{J}$, where $V_P$ denotes the exterior of $P$ considered as a knot in $V$. 
We refer to $V_P$ as the \emph{pattern space} associated with $P$. 
Details of these constructions can be found in \cite[Section 2]{GordonSatelliteKnot}.

Let $M$ be a manifold with torus boundary. A (rational) slope on the boundary $M$ is an element $[\alpha]$ of the projective space of $H_1(\partial M;\mathbb{Q})$, where $\alpha \in H_1(\partial M;\mathbb{Q}) \setminus \{ 0 \} $. Slopes can be identified as either
 \begin{itemize}
     \item a $\pm$-pair of primitive elements of $H_1(\partial M; \mathbb{Z})\equiv \fund{\partial M}$;
    \item a $\partial M$-isotopy class of essential simple closed curves on $\partial M$. 
 \end{itemize}
For further details, see \cite[Subsection 4.2]{boyer2022orderdetection}.

Let $\Sigma$ be a torus, and $\{\mu,\lambda\}$ a basis of $H_1(\Sigma;\mathbb{Z})$. We use the convention, depending on the choice of the base $\{\mu,\lambda\}$, such that an element $\nicefrac{p}{q} \in \mathbb{Q}\cup \{\nicefrac{1}{0}\}$ corresponds to the slope $p \mu + q \lambda \in H_1(\Sigma;\mathbb{Z})$. The \emph{distance} between two slopes $\nicefrac{p}{q}$ and $\nicefrac{r}{s}$ is $\Delta(\nicefrac{p}{q},\nicefrac{r}{s})=|ps-rq|$ and it corresponds to the absolute value of the algebraic intersection number between curves representing $\nicefrac{p}{q}$ and $\nicefrac{r}{s}$.

If $Y$ is a compact $3$-manifold with toroidal boundary components $\Sigma_1, \cdots, \Sigma_n$ with fixed bases $\{\mu_i,\lambda_i\}$ for each $H_1(\Sigma_i;\mathbb{Z})$, then
\[
    Y\left(\Sigma_1, \cdots, \Sigma_n;\nicefrac{r_1}{s_n},\cdots \nicefrac{r_n}{s_n}\right)
\]
denotes the closed $3$-manifold obtained by performing Dehn fillings along a simple closed curve representing $\nicefrac{r_i}{s_i}$ on $\Sigma_i$ for each $i=1,\cdots, n$. If $Y$ has a torus boundary component, we define the \emph{rational longitude} of $Y$ as the unique slope $\lambda_{Y} \subset Y$ such that $[\lambda_{Y}]$ is a torsion element of $H_1(Y; \mathbb{Z})$.

\begin{notation}
When the manifold $Y$ is the knot exterior $\compl{K}$,
we define the canonical basis $\{\mu_K,\lambda_K\} \subset \fund{\partial \compl{K}}$, where 
$\mu_K$ is given by the homotopy class of a curve that bounds an essential disc in $\nu(K) \subset S^3$ and $\lambda_K$ is the homotopy class of a curve that bounds a surface in $\compl{K}$, with orientations following the usual convention (i.e. meridional curve pushed into $\compl{K}$ and a longitudinal curve have linking number $+1$). We call these slopes knot \emph{meridian} and \emph{longitude} respectively. Sometimes the knot longitude is called \emph{null-homologous longitude}.
\label{notation: basis used}
\end{notation}
The rational longitude $\lambda_{E(K)}$ of $E(K)$ coincides as a slope with the longitude $\lambda_K$ we have just defined.
We express the slopes of $\compl{K}$ with respect this basis:
the fraction $\nicefrac{r}{s} \in \mathbb{Q}\cup \{\nicefrac{1}{0}\}$ corresponds to the slope $r \mu_K + s \lambda_K \in \fund{\partial \compl{K}}$.
With these coordinates, the slope $\nicefrac{1}{0}$ corresponds to the meridian $\mu_K$  and $\nicefrac{0}{1}$ to the longitude $\lambda_K$. We denote by $K(\nicefrac{r}{s})$ the Dehn filling on $\partial (\compl{K})$ along the curve that represents the slope $\nicefrac{r}{s}$.

\begin{rmk}\label{rmk mirror image}
Let $K$ be a knot and $\overline{K}$ its mirror image.
It is known that
    \[
        K(\nicefrac{r}{s})= r \left( \overline{K}(\shortminus \nicefrac{r}{s})\right),
    \]
    where $r(M)=-M$ denotes the manifold $M$ with opposite orientation. Thus,
    \[
        \fund{K(\nicefrac{r}{s})} = \fund{r \left(\overline{K}(\shortminus \nicefrac{r}{s})\right)}=\fund{\overline{K}(\shortminus \nicefrac{r}{s})}.
    \]
\end{rmk}

\section{General results for Satellite knots}
In this section, we establish some general results concerning satellite knots that will be used throughout this work. 
We introduce a partial ordering on knots via epimorphisms of their knot groups and study how $SU(2)$-abelian surgeries behave under this ordering.

 \begin{defn}\label{defn: K2 bigger than K1}
    Let $K_1,K_2$ be two knots in $S^3$. Let $\mu_{K_i},\lambda_{K_i} \in \fund{\partial \compl{K_i}}$ be, respectively, the classes of the meridian and the null-homologous longitude of $K_i$. We say that $K_2 \ge K_1$ if there exists an epimorphism
    \[
        \psi \colon \fund{\compl{K_2}} \twoheadrightarrow \fund{\compl{K_1}}
    \]
    such that the following are satisfied
    \begin{itemize}
        \item $\psi\left( \fund{\partial \compl{K_2}} \right) = \fund{\partial \compl{K_1}}$,
        \item Let $p$ and $q$ be two coprime integers. If $\psi\left( p \mu_{K_2}+ q \lambda_{K_2} \right)= \mu_{K_1}$, then $|p| =1 $ and $q=0$.
    \end{itemize}
\end{defn}

\begin{lemma}
    Let $K_1,K_2$ be two knots in $S^3$ such that $K_2 \ge K_1$. Let $\psi \colon \fund{\compl{K_2}} \twoheadrightarrow \fund{\compl{K_1}}$ be the epimorphism of Definition \ref{defn: K2 bigger than K1}.
    Let $\nicefrac{r}{s}\in \mathbb{Q}$ and $r',s' \in \mathbb{Z}$ be such that
        \[
            \psi\left( r \mu_{K_2} + s \lambda_{K_2}\right) = r' \mu_{K_1} + s' \lambda_{K_1} \in \fund{ \partial \compl{K_1}}.
        \]
        If the manifold $K_1(\nicefrac{r'}{s'})$ is not $SU(2)$-abelian, then the manifold $K_2(\nicefrac{r}{s})$ is not $SU(2)$-abelian.
        \label{lemma: K2 > K1, se K1 ha una chiurgia irr allora K2 la ha}
    \begin{proof}
        It is straightforward to see that since $\gcd(r,s)=1$, then $\gcd(r',s')=1$. Therefore the manifold $K_1(\nicefrac{r'}{s'})$ is well-defined.
        The map $\psi$ induces a surjective map
        \[
            \Psi \colon \fund{K_2(\nicefrac{r}{s})} \twoheadrightarrow \fund{K_1(\nicefrac{r'}{s'})}.
        \]
        By hypothesis, there exists an irreducible representation $\rho \colon \fund{K_1(\nicefrac{r'}{s'})}\to SU(2)$. The manifold $K_2(\nicefrac{r}{s})$ admits an irreducible $SU(2)$-representation by the following commutative diagram:
        \begin{equation*}
            \begin{tikzcd}
                \fund{\compl{K_2}}
                \arrow[d,twoheadrightarrow]
                \arrow[r,twoheadrightarrow,"\psi"]&  \fund{\compl{K_1}} 
                \arrow[d,twoheadrightarrow] & \\
                \fund{K_2(\nicefrac{r}{s})} \arrow[r,twoheadrightarrow," \Psi"] & \fund{K_1(\nicefrac{r'}{s'})} \arrow[r,"\rho"] & SU(2).
            \end{tikzcd}
        \end{equation*}
    \end{proof}
\end{lemma}

\begin{cor}\label{cor: K2 > K1, if K1 is not SU(2)abelian, then K2 is not}
    Let $K_1,K_2$ be two knots in $S^3$ such that $K_2 \ge K_1$. If the knot $K_1$ is not $SU(2)$-abelian, then the knot $K_2$ is not $SU(2)$-abelian as well.
    \begin{proof}
        Since $K_1$ is not $SU(2)$-abelian, $K_1(\nicefrac{r'}{s'})$ is $SU(2)$-abelian if and only if $\nicefrac{r'}{s'}=\nicefrac{1}{0}$.
        Let $\nicefrac{r}{s} \in \mathbb{Q}$ and $\psi \colon \fund{\compl{K_2}}\to \fund{\compl{K_1}}$ as in Definition \ref{defn: K2 bigger than K1}. Thus, the slope in $\partial \compl{K_2}$ represented by $\nicefrac{r}{s}$ is not mapped by $\psi$ into the knot meridian of $K_1$. With an abuse of notation we can say that $\psi(\nicefrac{r}{s})\neq \nicefrac{1}{0}$. Lemma \ref{lemma: K2 > K1, se K1 ha una chiurgia irr allora K2 la ha} implies the conclusion.
    \end{proof}
\end{cor}

Let $P$ be a satellite with winding number $w$ and let $K=P(J)$ be a satellite knot. The pattern space associated to $P$ has two boundary components:
\[
    \partial V_P = \partial \compl{J} \sqcup \partial \compl{K}.
\]
Let $\mathcal{U} \subset S^3$ be the unknot. The complement $\compl{P(\mathcal{U})}$ is the manifold
\[
    \compl{P(\mathcal{U})}= V_P \left(\partial \compl{J}; \lambda_{J}\right).
\]
Furthermore, let $\nicefrac{r}{s}\in \mathbb{Q}$, then the $\nicefrac{r}{s}$-filling of $\compl{P(\mathcal{U})}$ is
\[
     P(\mathcal{U})(\nicefrac{r}{s})=V_P \left(\partial \compl{J}, \partial \compl{K}; \lambda_{J}, \nicefrac{r}{s}\right).
\]
\begin{prop}\label{prop: satellite P(J) ammette una chirugia SU(2)-abelian then}
    Let $P$ be a satellite with winding number $w$ and let $K=P(J)$ be a satellite knot. Let $\nicefrac{r}{s} \in \mathbb{Q}$ be such that $K(\nicefrac{r}{s})$ is $SU(2)$-abelian. Then both of the following hold:
    \begin{itemize}
        \item The manifold $P(\mathcal{U})(\nicefrac{r}{s})$ is $SU(2)$-abelian.
        \item If $w \neq 0$, then the manifold $J(\nicefrac{r}{s w^2})$ is $SU(2)$-abelian.
    \end{itemize}
    \begin{proof}
        Let us suppose that $w \neq 0$. Let $\Sigma$ be the torus $\partial \compl{J}\subset \partial V_P$ in $\compl{K}$. The manifold $K(\nicefrac{r}{s})$ equals
        \[
            K(\nicefrac{r}{s})= \compl{J} \cup_{\Sigma} V_{P}\left(\partial \compl{K};\nicefrac{r}{s}\right).
        \]
        It is known that the rational longitude of $V_{P}\left(\partial \compl{K};\nicefrac{r}{s}\right)$ is 
        \[
        \lambda_C= \frac{1}{\gcd(r,sw^2)} \left(r \mu_J + sw^2 \lambda_J\right)= \frac{r}{sw^2}.
        \]
        Hence, \cite[Corollary 3.11]{Mino} implies that $\compl{J}(\lambda_C)=J(\lambda_C)=J(\nicefrac{r}{sw^2})$ is $SU(2)$-abelian.
        
        Let $\mathcal{U}$ denote the unknot in $S^3$. By \cite[Proposition 3.4]{KnotGroupEpimo}, there exists an epimorphism
        \[
        \psi: \fund{\compl{P(K)}} \twoheadrightarrow \fund{\compl{P(\mathcal{U})}}
        \]
        such that $\psi(\mu_{P(K)})=\mu_{P(K)}$ and $\psi(\lambda_{P(K)})=\lambda_{P(\mathcal{U})}$. Hence, $P(K) \ge P(\mathcal{U})$. The conclusion is given by Lemma \ref{lemma: K2 > K1, se K1 ha una chiurgia irr allora K2 la ha}.
    \end{proof}
\end{prop}

For the sake of completeness, we mention that the second part of Proposition \ref{prop: satellite P(J) ammette una chirugia SU(2)-abelian then} is proven in \cite{SU2AverseKnots} in a different way.

Let $T \subset S^3$ be a non-trivial torus knot. 
We say that the cable knot $C_{p_1,q_1}(T)$ is a $1$-iterated torus knot. Similarly, for $n \ge 2$, we call $n$-iterated torus knot the cable knot $C_{p_n,q_n}(J)$ where $J$ is an $(n-1)$ iterated torus knot.
In general, we say that $K$ is an iterated torus knot if $K$ is either a torus knot or an $n$-iterated torus knot, with $n \ge 1$. 
From the $SU(2)$ perspective, the $SU(2)$-abelian surgeries of iterated torus knots are classified in the following:
\begin{teo}[{\cite{Sivek_Zentner}}]\label{thm:iterated-cables}
Let $K$ be an iterated torus knot.  If some nontrivial $r$-surgery on $K$ is $SU(2)$-abelian, then $K$, $r$, and $K(r)$ are among the following:
\begin{align*}
K &= T_{p,q}: & r&=pq+\tfrac{1}{m}\ (m \neq 0), & K(r) &= L(mpq+1,mq^2) \\
K &= T_{p,2\epsilon}: & r&=2\epsilon p, & K(r) &= L(p,2\epsilon) \# \mathbb{RP}^3 \\
K &= C_{2pq+\epsilon,2}(T_{p,q}): & r &= 4pq+\epsilon, & K(r) &= L(4pq+\epsilon,4q^2) \\
&& \mathrm{or\ } r &= 4pq+2\epsilon, & K(r) &= L(2pq+\epsilon,2q^2) \# \mathbb{RP}^3.
\end{align*}
Here $\epsilon$ can be either $+1$ or $-1$.  In particular, $n$-iterated torus knots with $n \geq 3$ are not $SU(2)$-abelian.
\end{teo}

\section{Review of the Pillowcase}
\label{sec: pillowcase}
We refer the reader to \cite{Mino} and \cite{ThePillowcaseHeddenHeraldKirk} for background material.
Let $G$ be a group. We denote by $\mathcal{R}(G)$ the space $\mbox{Hom}(G,SU(2))$. The group $SU(2)$ acts on $\mathcal{R}(G)$ by conjugation. We define
\[
    \chi(G) \coloneqq \left. \mbox{Hom}\left( G, SU(2)\right) \right/_{\mbox{conjugation}} = \left. \mathcal{R}(G)\right/_{\mbox{conjugation}},
\]
and 
\[
q \colon \mathcal{R}(G) \to \chi(G)
\]
be the quotient map.
Let $M$ be a manifold. If $G = \fund{M}$, then we simplify the notation by writing $\mathcal{R}(M)$ and $\chi(M)$ instead of $\mathcal{R}(\fund{M})$ and $\chi(\fund{M})$. Let $K$ be a knot in $S^3$. If $M=\compl{K}$, then we denote by $\mathcal{R}(K)$ and $\chi(K)$ the spaces $\mathcal{R}(\fund{\compl{K}})$ and $\chi(\fund{\compl{K}})$.

Let $\Sigma$ be a torus. We define $\R{\Sigma}$ as $\text{Hom}\left( \fund{\Sigma},\Ui \right)$,
where $\Ui$ is the subgroup of diagonal matrices of $SU(2)$.
There exists a natural inclusion $\R{\Sigma} \subset \mathcal{R}(\Sigma)$.
The space $\R{\Sigma}$ is homeomorphic to the torus $S^1 \times S^1 =[0,2\pi]^2/_{\sim} $. Every representation in $\mathcal{R}(\Sigma)=\mbox{Hom}(\fund{\Sigma},SU(2))$ is conjugate to one representation in $\R{\Sigma}=\mbox{Hom}(\fund{\Sigma},U(1))$. This implies that 
\[
    \restr{q}{\R{\Sigma}} \left( \R{\Sigma} \right) = \chi (\Sigma).
\]

Let $K \subset S^3$ be a knot, and let $\mu_K$ and $\lambda_K$ be the classes of the meridian and the null-homologous longitude in $\fund{ \partial \compl{K}}$ as in Section \ref{sec: notation}.
We give coordinates to the space $\R{\partial \compl{K}}$ with respect to  the basis $\{\mu_K , \lambda_K\} \subset \fund{\partial \compl{K}}$ in the following way:
the point $(\theta, \psi) \in [0,2\pi]^2/_{\sim} = \R{\partial \compl{K}}$ corresponds to the unique representation
\begin{equation}
\fund{\partial \compl{K}}\to \Ui, \quad \text{with }\quad \mu_K \mapsto \begin{bmatrix}
    e^{i\theta} & 0 \\ 0 & e^{-i\theta}
\end{bmatrix} \quad \text{and} \quad \lambda_K \mapsto \begin{bmatrix}
    e^{i\psi} & 0 \\ 0 & e^{-i\psi}
\end{bmatrix}.
\label{eq: coordinates of R(partial E(K))}
\end{equation}

\begin{defn}\label{defn: linee p/q}
    Let $\R{\partial \compl{K}}$ be parameterized as in \eqref{eq: coordinates of R(partial E(K))}. Let $p,q \in \mathbb{Z}$ be two coprime integers. We define the lines $ L_{\nicefrac{p}{q}}^0$ and $ L_{\nicefrac{p}{q}}^\pi$ as
    \[
        L_{\nicefrac{p}{q}}^\varepsilon \coloneqq \left\{ (\theta,\psi) \in \R{\partial \compl{K}} \middle| p \theta + q\psi \equiv_{2\pi} \varepsilon \right\} \quad \text{with} \quad \varepsilon \in \{0,\pi\}.
    \]
    We will say that such lines have slope $-\nicefrac{p}{q}.$
\end{defn}

We define $T(\compl{K},\partial \compl{K}) \subset \R{\partial \compl{K}}$ as:
\[
     T(\compl{K},\partial \compl{K}) \coloneqq \left\{  \eta \in \R{\partial \compl{K}} \, \middle| \, \exists \rho \in \mathcal{R}( \compl{K}) \text{ such that } \restr{\rho}{\fund{\partial \compl{K}} }\equiv \eta \right\}.
\]

Equivalently, the space $T(\compl{K},\partial \compl{K})$ is the space of representations $\fund{\partial \compl{K}}\to \Ui$ that extend to $\fund{\compl{K}}$. 
In order to make the notation a little lighter we denote by $T(K,\partial)$ the space $T(\compl{K},\partial \compl{K})$.
We define $H_{K} \subset T(K,\partial)$ as
\[
    H_{K} \coloneqq \left\{ \eta \in \R{\partial \compl{K}} \,\middle|\, \exists \rho \in \mathcal{R}(\compl{K}) \text{ such that } \restr{\rho}{\fund{\partial \compl{K}}}\equiv \eta\text{ and $\rho$ is irreducible} \right\}.
\]
Equivalently, a representation $\fund{\partial \compl{K}}\to \Ui$ is in $H_{K}$ if and only if it extends to an irreducible representation $\fund{\compl{K}} \to SU(2)$.

Lastly, we define $A_K \subset T(K,\partial)$ as
\[
    A_{K} \coloneqq \left\{ \eta \in \R{\partial \compl{K}} \,\middle|\, \exists \rho \in \mathcal{R}(\compl{K}) \text{ such that } \restr{\rho}{\fund{\partial \compl{K}}}\equiv \eta\text{ and $\rho$ is abelian} \right\}.
\]
As a consequence of \cite[Corollary 3.9]{Mino} we obtain that
\[
    A_{K} \coloneqq \left\{ \eta \in \R{\partial \compl{K}} \,\middle|\, \eta(\lambda_K)=1 \right\},
\]
where $\lambda_K \subset \partial \compl{K}$ is the null-homologous longitude of the knot $K$.

\begin{rmk}\label{rmk: surgery and the line}
    Let $K$ be a non-trivial knot.
    Let $\R{\partial \compl{K}}$ be given with coordinates $(\theta,\psi)$ as in \eqref{eq: coordinates of R(partial E(K))}.
    As a consequence of \cite[Theorem 3.7]{Mino}, the surgery $K(\nicefrac{r}{s})$ is $SU(2)$-abelian if and only if $H_K \cap L_{\nicefrac{r}{s}}^0$ is empty. From Definition \ref{defn: linee p/q}, we obtain 
    \[
        L_{\nicefrac{r}{s}}^0 = \bigcup_{{j \in \{1,\cdots,|r|\}}} \left\{ \left((\theta, -\frac{r}{s}\theta + \frac{2\pi j}{r}\right) \middle| \theta_K \in [0,2\pi] \right\} \subset \R{\partial \compl{K}}.
    \]
    Hence, $K(\nicefrac{r}{s})$ is $SU(2)$-abelian if and only if
    \[
        \left(\theta, -\frac{r}{s}\theta + \frac{2\pi j}{q}\right) \notin H_K
    \]
    for every $\theta \in [0,2\pi]$ and $j \in \{1,\cdots, |r|\}$.
\end{rmk}

Let $\rho_1$ and $\rho_2$ be the representations in $\R{\partial \compl{K}}$ corresponding to the points $(\theta,\psi)$ and $(2\pi-\theta,2\pi - \psi)$. Let $X=\begin{bmatrix}
        0 & -1 \\ 1 & 0
    \end{bmatrix} \in SU(2)$. We notice that
\[
    X
     \begin{bmatrix}
        e^{i\alpha} & 0 \\ 0 & e^{-i\alpha}
    \end{bmatrix}
    X^{-1}
    = 
    \begin{bmatrix}
        e^{-i\alpha} & 0 \\ 0 & e^{i\alpha}
    \end{bmatrix}
    = 
    \begin{bmatrix}
        e^{i(2\pi -\alpha) } & 0 \\ 0 & e^{-i(2\pi - \alpha)}
    \end{bmatrix}.
\]
This implies that $X \rho_1 X^{-1}=\rho_2$. Hence, the representations $\rho_1$ and $\rho_2$ are in the same class in $\chi(\partial \compl{K})$. It can be proven that
\[
    \chi(\partial \compl{K}) = \frac{\R{\partial \compl{K}}}{\text{conjugation}}= \frac{[0,2\pi] \times [0,2\pi]}{(\theta,\psi) \sim (2\pi -\theta,2\pi -\psi)}.
\]
Equivalently, the space $\chi(\partial \compl{K})$ can be obtained as the quotient of the fundamental domain $[0, \pi ] \times [0, 2\pi]$ by the following identities:
\[  
(0,\psi) \sim (0,2\pi - \psi), \quad (\theta,0) \sim (\theta,2\pi), \quad \text{and} \quad (\pi,\psi) \sim (\pi,2\pi - \psi).
\]
Let the space $\R{\partial \compl{K}}$ be given with coordinates as in \eqref{eq: coordinates of R(partial E(K))}. The map $q\colon \R{\partial \compl{K}} \to \chi (\partial \compl{K})$ induces the following parametrization of the target: the point $(\theta,\psi) \in [0,\pi]\times [0,2\pi] /_\sim$ corresponds to the $SU(2)$ conjugacy class of the representation
\begin{equation}
\fund{\partial \compl{K}}\to \Ui, \quad \text{with }\quad \mu_K \mapsto \begin{bmatrix}
    e^{i\theta} & 0 \\ 0 & e^{-i\theta}
\end{bmatrix} \quad \text{and} \quad \lambda_K \mapsto \begin{bmatrix}
    e^{i\psi} & 0 \\ 0 & e^{-i\psi}
\end{bmatrix}.
\end{equation}
The space $\chi(\partial \compl{K})$ is homeomorphic to a sphere with four singular points of order $2$ and it is sometimes called the \emph{pillowcase}.
The inclusion $\iota \colon \partial \compl{K} \to \compl{K}$ induces the restriction
\[
    \iota^\ast\colon \chi(K) \to \chi(\partial \compl{K}).
\]
The image $\iota^\ast\chi(K) \subset \chi(\partial \compl{K})$ is sometimes called the \emph{image of the knot in the pillowcase}. It is easy to see that
\[
    q \left( T(K,\partial)\right) = \iota^\ast\chi(K) \subset \chi(\partial \compl{K})
\]
and
\[
    T(K,\partial) =q^{-1} \left( \iota^\ast\chi(K) \right) \subset \R{\partial \compl{K}}.
\]

\begin{cor}
    Let $K$ be a non-trivial knot. Then any topologically embedded path from the point $P = (0, \pi)$ to $Q = (\pi, \pi)$ in $\R{\partial \compl{K}}$ and missing the line 
    \[
    \left\{\psi \equiv_{2\pi} 0\right\} \subset \R{\partial \compl{K}}
    \]
    has an intersection point with $H_{K} \subset T(K,\partial)$.
    \label{cor: every walk in R(partial K) from P to Q}
    \begin{proof}
        We remind the reader that $q\colon \R{\partial \compl{K}} \to \chi(\partial \compl{K})$ is the natural quotient map.
        Let $\gamma \subset \R{\partial \compl{K}}$ be a walk from $P$ to $Q$ and missing the line $\{\psi \equiv_{2\pi} 0\}$. The image $q(\gamma)$ is an immersed walk in the pillowcase $\chi(\partial \compl{K})$ from $q(P)=(0,\pi)$ to $q(Q)=(\pi,\pi)$ missing the line $\left\{\psi \equiv_{2\pi } 0 \right\} \subset \chi(\partial \compl{K})$.
        Let $\gamma' \subset \chi(\partial K)$ be a sub-walk of $q(\gamma)$ that is embedded and contains the points $q(P)$ and $q(Q)$.
        Let $(\theta,\psi)$ be a point in $\iota^\ast \chi (K)\cap \gamma' \subset \chi(K)$. This point exists by \cite[Theorem 7.2]{IntegerHomSL2CIrrResp}. In particular $\psi \nequiv_{2\pi} 0$. We obtain that the pre-image $q^{-1}(\theta,\psi) \subset \R{\partial K}$ contains a point in
        \[
            \gamma \cap q^{-1} \left( \iota^\ast\chi(K) \right) = \gamma \cap T(K,\partial).
        \]
        Since $\psi \nequiv_{2\pi} 0$, we obtain that this point is in $\gamma \cap H_{K}$ by \cite[Lemma 2.8]{SU2cyclicSurgeryFormula}.
    \end{proof}
\end{cor}

\begin{defn}
    Let $K\subset S^3$ be a non-trivial knot. Let $\gamma:(-1,1) \to H_K$ be an open arc with
    \[
        \lim_{x \mapsto \pm 1} \gamma(x) \in A_K.
    \]
    Let $\Gamma$ be the support of $\gamma$.
    We say that $\Gamma$ is a \emph{good arc} if the map
    \[
        \fund{\Gamma \cup A_K} \longrightarrow \fund{\R{\partial \compl{K}}}
    \]
    induced by the inclusion $\Gamma \cup A_K \subset \R{\partial \compl{K}}$ is surjective.
\end{defn}
Roughly speaking, we say that $\Gamma$ is a good arc if it cannot be pushed into the set $A_K \subset \R{\partial \compl{K}}$ by a homotopy of the torus $\R{\partial \compl{K}}$ that does not move $A_K$.

The magenta arcs in Figure \ref{figure: trefoild mu5} and Figure \ref{figure: figure eight knot} are examples of good arcs for the trefoil and the figure eight knot.
Corollary \ref{cor: every walk in R(partial K) from P to Q} implies that every knot in $S^3$ admits a good arc.

\section{Composite knots}

In this section, we develop Theorem \ref{teo: double knot}, one of the main result of this work. 
This suggests that every non-trivial surgery on a composite knot 
produces a non-$SU(2)$-abelian $3$-manifold.

It is well known that an integer surgery along a composite knot
yields a toroidal $3$-manifold that is obtained by gluing together the two knot exteriors along their boundary tori.
The following lemma makes the gluing diffeomorphism explicit.

\begin{lemma}\label{lemma: int surgery on composite knot}
Let $K_1$ and $K_2$ be two nontrivial knots of $S^3$ and $K_\# $ the composite knot $ K_1 \# K_2$.
For $i \in\{1,2\}$ we denote by $\{\mu_i,\lambda_i\}$ the system of meridian and null-homologous longitude for $K_i$.
Let $r \in \mathbb{Z}$. The surgery $K_\#(r)$ is the toroidal manifold
        \[
            \compl{K_1} \cup_{\varphi} \compl{K_2},
        \]
        where $\varphi \colon \partial \compl{K_1} \to \partial \compl{K_2}$ is an orientation reversing diffeomorphism such that
        \[
        \varphi(\mu_1) = \mu_2 \quad \text{and} \quad \varphi(\lambda_1)= \mu_2^{-r}\lambda_2^{-1}.
        \]
\begin{proof}
    For $i \in\{1,2,\#\}$ we denote by $\{\mu_i,\lambda_i\}$ the system of meridian and null-homologous longitude for $K_i$.
    According to \eqref{eq: presentazione composite knot},
    \[
        \fund{K_\#(r)}= \frac{\fund{\compl{K_\#}}}{\normalsubgroup{\mu_\#^r\lambda_\#}} = \frac{\fund{\compl{K_1}\ast \fund{\compl{K_2}}}}{\normalsubgroup{\mu_\#^r\lambda_\#, \mu_\#=\mu_1=\mu_2, \lambda_\#=\lambda_1\lambda_2}}.
    \]
    Therefore in $\fund{K_\#(r)}$ we have that $\mu_1=\mu_2$ and $\mu_2^r\lambda_2=\lambda_1^{-1}$. This implies that the groups $\fund{K_\#(r)}$ and $\fund{\compl{K_1}\cup_{\varphi}\compl{K_2}}$ are isomorphic and therefore we obtain the conclusion.
\end{proof}
\end{lemma}

Lemma \ref{lemma: int surgery on composite knot} is proven with a different technique
in \cite[Lemma 1.3]{sorya2026numbershareddehnsurgeries} by my doctoral colleague and friend Patricia Sorya.
Now we need to understand how the homomorphism $\varphi \colon \partial \compl{K_1} \to \partial \compl{K_2}$
acts on the representation spaces $\R{\partial\compl{K_1}}$ and $\R{\partial\compl{K_2}}$.

\begin{prop}
    Let $K_1,K_2$ be two non-trivial knots in $S^3$ and let $K_\# = K_1 \# K_2$ be their composite knot. Let $r \in \mathbb{Z}$.
    Let $\varphi \colon \partial \compl{K_1} \to \partial \compl{K_2}$ be the orientation reversing diffeomorphism defined in Lemma \ref{lemma: int surgery on composite knot}.
    For $i \in \{1,2\}$, we parameterize the space $\R{\partial \compl{K_i}}$ with coordinates $(\theta_i,\psi_i)$ as in \eqref{eq: coordinates of R(partial E(K))}.
    Then the induced map
    \[
        \varphi^\ast \colon \R{\partial \compl{K_2}} \to \R{\partial \compl{K_1}}, \quad \eta \mapsto \eta \circ \varphi_\ast
    \]
    is given by the formula
    \[
        \varphi^\ast(\theta_2,\psi_2) = (\theta_2, -r\theta_2 - \psi_2) \in \R{\partial \compl{K_1}}.
    \]
    \begin{proof}
        Let $\eta_2 \in \R{\partial \compl{K_2}}$ be the representation corresponding
        to the point $(\theta_2,\psi_2) \in [0,2\pi]^2/_{\sim}$.
        Let $\eta_1 \in \R{\compl{K_1}}$ be the image $\varphi^\ast(\eta_2) = \eta_1$,
        and $(\theta_1,\psi_1) \in [0,2\pi]^2/_{\sim}$ its corresponding point. We prove the formula
        for $\varphi^\ast$ by computing the point $(\theta_1,\psi_1)$ in terms of $(\theta_2,\psi_2)$.

        Since $\varphi(\mu_1) = \mu_2$ and $\varphi(\lambda_1)= \mu_2^{-r}\lambda_2^{-1}$, we have
        \[
            \eta_1(\mu_1) = \eta_2(\varphi(\mu_1)) = \eta_2(\mu_2) = e^{i\theta_2}
        \]
        and
        \[
            \eta_1(\lambda_1) = \eta_2(\varphi(\lambda_1)) = \eta_2(\mu_2^{-r}\lambda_2^{-1}) = (\eta_2(\mu_2))^{-r} (\eta_2(\lambda_2))^{-1} = e^{i(-r\theta_2 - \psi_2)}.
        \]
        Therefore, the homomorphism $\eta_1=\varphi^\ast(\eta_2)$ corresponds to the point $(\theta_2,-r\theta_2 - \psi_2)$. This implies that $\varphi^\ast(\theta_2,\psi_2) = (\theta_2, -r\theta_2 - \psi_2)$.
        \end{proof}
    \label{prop: action on integer surgery on T(K)}    
\end{prop}

\begin{defn}
Let $K\subset S^3$ be a knot and let $\Delta_K(t)$ be its Alexander polynomial. We say that $K$ is \emph{SU(2)-clean} if for every $r\in\mathbb{Z}$ such that the surgery $K(r)$ is $SU(2)$-abelian, we have $
\Delta_K(\zeta_r^k) \neq 0 $ for every $k\in \{1,\dots,r\}$
where $\zeta_r = e^{\frac{2\pi i}{r}}$.
\label{defn SU(2)-clean}
\end{defn}

Definition \ref{defn SU(2)-clean} depends on the choice of basis $\{\mu,\lambda\}$ we made in Section \ref{sec: notation};
however, since we have fixed a choice of meridian and longitude,
this definition is not ambiguous. Making it independent of this
choice is beyond the scope of this work.

\begin{cor}\label{cor: some examples}
    The following knots are $SU(2)$-clean:
    \begin{itemize}
        \item The unknot $\mathcal{U}$,
        \item torus knots $T(p,q)$,
        \item $2$-bridge non-torus knots,
        \item the pretzel knots $P(-2,3,7)$ and $P(-2,3,13)$.
    \end{itemize}
    \begin{proof}
Since the Alexander polynomial of the unknot is $\Delta_{\mathcal{U}}(t)=1$, the unknot is a $SU(2)$-clean knot.

By Theorem \ref{thm:iterated-cables}, the only integral $SU(2)$-abelian surgeries on a non-trivial torus knot $T(p,q)$ are the slopes
$r = pq \pm 1$.
It is known that the Alexander polynomial of such a torus knot $T(p,q)$ is
\[
\Delta_{T(p,q)}(t) = \frac{(t^{pq}-1)(t-1)}{(t^p-1)(t^q-1)}.
\]
This implies that the roots of $\Delta_{T(p,q)}(t)$ are $pq$-th roots of unity.
Therefore, since $\gcd(pq,pq\pm1)=1$, we obtain that $\Delta_{T(p,q)}(\zeta_r^k) \neq 0$ for every $k=1,\dots,pq$.
Hence, every torus knot is $SU(2)$-clean.

By \cite[Theorem 1.1]{bowden2007winding}, if $K$ is $2$-bridge knot that is not a torus knot,
then there are no integer $SU(2)$-abelian surgeries on $K$. Hence, $K$ is $SU(2)$-clean.

The Alexander polynomial of the pretzel knot $P(-2,3,7)$ is
\[
\Delta_{P(-2,3,7)}(t) = t^{10}-t^9+t^7-t^6+t^5-t^4+t^3-t+1.
\]
A direct computation shows that $\Delta_{P(-2,3,7)}(t)$ is irreducible over $\mathbb{Q}$
and does not coincide with any cyclotomic polynomial $\Phi_k(t)$. 
Since the minimal polynomial of any root of unity over $\mathbb{Q}$ is a cyclotomic polynomial, it follows that $\Delta_{P(-2,3,7)}(t)$
has no root of unity among its roots.
Consequently $\Delta_{P(-2,3,7)}(\zeta_r^k)\neq0$ holds automatically for every $r \in \mathbb{Z}$ and every $k \in \{1, \cdots, r\}$,
regardless of which integer surgeries on $P(-2,3,7)$ are $SU(2)$-abelian.
This implies that the pretzel knot $P(-2,3,7)$ is $SU(2)$-clean.

The Alexander polynomial of the pretzel knot $P(-2,3,13)$ is
\[
\Delta_{P(-2,3,13)}(t) = t^{16}-t^{15}+t^{13}-t^{12}+t^{11}-t^{10}+t^9-t^8+t^7-t^6+t^5-t^4+t^3-t+1.
\]
A similar argument shows that $\Delta_{P(-2,3,13)}(t)$ has no root of unity among its roots, and hence $P(-2,3,13)$ is also $SU(2)$-clean.
\end{proof}
\end{cor}

The following theorem appears to be well known to the author, although no explicit reference could be found. For the reader’s convenience, we include a proof.
This result provides some justification for the assumption that the knot under consideration is $SU(2)$-clean in our study of composite knots.
As a consequence of Proposition \ref{prop: satellite P(J) ammette una chirugia SU(2)-abelian then}, if a knot $K_1$ is not $SU(2)$-abelian, then the connected sum $K_1 \# K_2$ is not $SU(2)$-abelian for any knot $K_2$.
Therefore, one is naturally led to consider pairs of knots for which each factor admits a non-trivial $SU(2)$-abelian surgery.
This is precisely the setting in which the $SU(2)$-clean assumption becomes relevant.
Indeed, $SU(2)$-abelian knots may be regarded as the $SU(2)$ analogue of knots admitting a cyclic surgery.
In particular, any knot admitting a non-trivial cyclic surgery admits a non-trivial
$SU(2)$-abelian surgery. It is not known whether the converse holds,
or whether there exists a knot admitting a non-trivial $SU(2)$-abelian surgery
and only the trivial cyclic surgery.
Since cyclic surgeries are necessarily integral by the celebrated Cyclic Surgery Theorem of Culler, Gordon, Luecke, and Shalen,
the analogy is especially close for integer surgeries.

\begin{teo}
\label{thm:berge-clean}
Let $K\subset S^3$ be a knot and let $p\in\mathbb{Z}_{>0}$ be such that $K(p)$ is a lens space. Then
\[
\Delta_K(\zeta_p^k)\neq 0\qquad\text{for every }k=1,\dots,p.
\]
\begin{proof}
We recall that $\zeta_p^k \coloneq e^{i\frac{2 \pi k}{p}}$.
Fix $k\in\{1,\dots,p-1\}$ and let $\chi_k:\mathbb{Z}/p\mathbb{Z}\to\mathbb{C}^{*}$ be the character $1 \mapsto \zeta_p^{k}$.
It is known that for every knot $K \subset S^3$ and $\epsilon \in \{\pm 1\}$, we have that $\Delta_K(\epsilon) \neq 0$. Hence, we assume that $\zeta_p^k \neq \pm 1$.
By Milnor's Dehn surgery formula for Reidemeister torsion \cite{Milnor1962}, the torsion of $K(p)$ twisted by $\chi_k$ is given, up to multiplication by a unit of $\mathbb{Z}[\zeta_p]$, by
\[
\tau\big(K(p),\chi_k\big)\ \doteq\ \frac{\Delta_K(\zeta_p^{k})}{(\zeta_p^{k}-1)(\zeta_p^{-k}-1)}.
\]
The denominator is nonzero since $\zeta_p^k\neq1$ for $1\le k\le p-1$.

Since $K(p)=L(p,q)$ is a lens space, its Reidemeister torsion with respect to any nontrivial character of $\pi_1(L(p,q))\cong\mathbb{Z}/p\mathbb{Z}$ is classically known to be well defined and nonzero \cite{Reidemeister1935}.

Combining the two computations of $\tau(K(p),\chi_k)$, we conclude that $\Delta_K(\zeta_p^{k})\neq0$. Since $k\in\{1,\dots,p-1\}$ was arbitrary, we get the conclusion.
\end{proof}
\end{teo}

The following proposition is a consequence of \cite[Theorem 19]{EpKlassenSu2} and the corresponding discussion on page $826$, and it will be used in the proof of Proposition \ref{prop: integer surgery on composite}.
\begin{prop}[{\cite[Theorem 19]{EpKlassenSu2}}]
    \label{prop: limit prop of Klassen}
    Let $K \subset S^3$ be a non-trivial knot.
    Let us parameterize the space $\R{\partial \compl{K}}$ with coordinates $(\theta,\psi)$ as in \eqref{eq: coordinates of R(partial E(K))}.
    If a sequence $\{\eta_n\}_{n \in \mathbb{N}}$ of points in $H_K \subset T(K,\partial)$ converges to a point $\eta_\infty = (\alpha,0)  \in A_K$,
    then $\Delta_K(e^{2i\alpha})=0$. Here $\Delta_K(t)$ is the Alexander polynomial of the knot $K$.
\end{prop}

\begin{prop}\label{prop: integer surgery on composite}
    Let $K\subset S^3$ be a non-trvial $SU(2)$-clean knot.
    Every integer surgery on the composite knot $K \# K$ is not $SU(2)$-abelian.
    \begin{proof}
        Let $K_1=K$ and $K_2=K$ be the two copies of $K$ and $K_\#=K_1 \# K_2$.
        Let $r \in \mathbb{Z}$. We will prove that the surgery $K \# K(r)$ is not $SU(2)$-abelian.
        Let $\varphi \colon \partial \compl{K_1} \to \partial \compl{K_2}$ be the diffeomorphism as in Lemma \ref{lemma: int surgery on composite knot},
        therefore
        \[
            K_\#(r)=\compl{K_1} \cup_{\varphi} \compl{K_2}.
        \]
        Let
        $\varphi^\ast \colon \R{\partial \compl{K_2}} \to \R{\partial \compl{K_1}}$ be the induced map as in Proposition \ref{prop: action on integer surgery on T(K)}.
        In particular, $\varphi^\ast(T(K_2,\partial))$ is a subset of $\R{\compl{K_1}}$.
        By \cite[Theorem 3.7]{Mino}, the surgery $K_\#(r)$ is $SU(2)$-abelian if and only if
        \[
            H_{K_1} \cap \varphi^\ast(H_{K_2}) = H_{K_1} \cap \varphi^\ast(A_{K_2}) = A_{K_1} \cap \varphi^\ast(H_{K_2}) = \emptyset.
        \]
        
        Let us recall that, since $K_1=K_2$, the spaces $\R{\compl{K_1}}$ and $\R{\compl{K_2}}$ are the same,
        and therefore the map $\varphi^\ast$ is a self-map.
        If there exists a point $ x \in H_{K_2}$ that is a fixed point for $\varphi^\ast$, meaning
        $x=\varphi^\ast(x)$, then
        \[
            x \in H_{K_1} \cap \varphi^\ast(H_{K_2}) = H_{K} \cap \varphi^\ast(H_{K}) \subset \R{\compl{K_1}}= \R{\compl{K}}.
        \]
        Thus, this intersetion is not empty and the surgery $K_\#(r)$ is not $SU(2)$-abelian by \cite[Theorem 3.7]{Mino}.
        We will show the conclusion by proving that there exists such a fixed point $x \in H_{K_2} = H_{K}$.

        Now we change the coordinates: we parameterize the space $\R{\partial \compl{K}}$ with coordinates
        $(\theta,\psi)$ as in \eqref{eq: coordinates of R(partial E(K))}.
        Proposition \ref{prop: action on integer surgery on T(K)} implies that the
        fixed points of $\varphi^\ast$ are the points 
        \[ 
            (\theta,\psi) \in \R{\partial \compl{K}} \quad \text{such that} \quad r \theta + 2 \psi = 0.
        \]
        
        If $K(r)$ is not $SU(2)$-abelian, then as an application of Proposition \ref{prop: satellite P(J) ammette una chirugia SU(2)-abelian then}, the surgery $K_{\#}(r)$ is not $SU(2)$-abelian and we are done.
        We can therefore suppose that $K(r)$ is $SU(2)$-abelian.
        By Remark \ref{rmk: surgery and the line}, this implies that the line $L_r^0 = \{(\theta,\psi) \in \R{\partial \compl{K}} \mid r\theta + \psi = 0\}$ does not intersect the space $H_K$.
        By Corollary \ref{cor: every walk in R(partial K) from P to Q},
        there exists a good arc $\Gamma \subset H_K$ that connects the points
        \[
            (\alpha,0) \in \R{\compl{K}} \quad \text{and} \quad (\beta,2\pi) \in \R{\partial \compl{K}}.
        \]
        Without loss of generality, we can suppose that $\Gamma$ does not intersect the line $L_r^0$, because if it did $K(r)$ would not be $SU(2)$-abelian.
        Since the knot $K$ is $SU(2)$-clean, we have that $\Delta_K(e^{\frac{2\pi i n}{r}}) \neq 0$ and for every $n \in \{1, \cdots, r\}$.
        By Proposition \ref{prop: limit prop of Klassen}, it does not exist
        a $k \in\{1, \cdots, r\}$ such that 
        $\alpha$ (resp. $\beta$) is equal to $\frac{2\pi k}{r}$.
        Therefore there exists a $k \in \{1, \cdots, r-1\}$ such that
        \[
            \alpha \in \left( \frac{2\pi k}{r}, \frac{2\pi (k+1)}{r}\right) \quad \text{and} \quad \beta \in \left( \frac{2\pi (k+1)}{r}, \frac{2\pi (k+2)}{r}\right).
        \]
        Figure \ref{figure: good arc for SU(2)-clean knot} shows the situation we are handling: here the line
        $L_r^0$ is drawn in blue and the good arc $\Gamma \subset H_K$ is drawn in magenta.

        Let $S \subset \R{\compl{K}}$ be the segment, connecting the points
        \begin{equation}
            \label{eq: endpoints of S}
            \left(\frac{2\pi k}{r}, 0\right), \left(\frac{2\pi (k+2)}{r}, 2\pi \right) \subset L_r^0\cap \{\psi=0\} \subset \R{\compl{K}},
        \end{equation}
        such that the interior of $S$ does not intersect the union $L_r^0 \cup \{\psi=0\}$.
        The segment $S$ is
        represented in Figure \ref{figure: good arc for SU(2)-clean knot} in yellow, and is defined as
        \[
            S = \left\{ (\theta,\psi) \in \R{\compl{K}} \mid | \theta \in \left[ \frac{ 2\pi k}{r}, \frac{2\pi (k+2)}{r}\right], \psi = - \frac{r}{2} \theta + (k+2) \pi \right\}.
        \]
        Therefore $S$ contains fixed points of $\varphi^\ast$.
        Since $\alpha$ and $\beta$
        lie strictly between the intervals described in \eqref{eq: endpoints of S},
        the arc $\Gamma$ must cross the segment $S$. Therefore there exists a point
        $x \in \Gamma \cap S$, as shown in Figure \ref{figure: good arc for SU(2)-clean knot}.
        This point $x$ is a fixed point of $\varphi^\ast$ and it lies in $H_K$, therefore, as we explained above,
        the intersection $H_K \cap \varphi^\ast(H_K)$ is not empty and the surgery $K \# K(r)$ is not $SU(2)$-abelian.
    \end{proof}
\end{prop}

Let $K_1$ and $K_2$ be two knots in $S^3$. We define $K_\# \subset S^3$ as the composite knot $K_1\#K_2$. For $i \in \{1,2,\#\}$, let $\{\mu_i,\lambda_i\}$ be an ordered basis of meridian and longitude for $K_i$. We recall the presentation \eqref{eq: presentazione composite knot}:
\[
    \fund{\compl{K_\#}} = \frac{\fund{\compl{K_1}}\ast \fund{\compl{K_2}}}{\normalsubgroup{\mu_1=\mu_2=\mu_\# , \lambda_\#=\lambda_1 \lambda_2}}.
\]
For $i \in \{1,2,\#\}$, we equip the space $\R{\partial \compl{K_i}}$ with coordinates $(\theta_i,\psi_i)$ as in \eqref{eq: coordinates of R(partial E(K))} with respect to the ordered basis $\{\mu_i,\lambda_i\}$.
We define the subset of $\R{\partial \compl{K_\#}}$
\begin{align}
\label{eq: description T(K) composite knot}
    T \coloneqq \left\{(\theta_\#,\psi_\#) \middle| \theta_\# =\theta_1=\theta_2, \psi_\#=\psi_1+\psi_2, (\theta_1,\psi_1) \in T(K_1,\partial), (\theta_2,\psi_2) \in T(K_2,\partial) \right\}.
\end{align}
\begin{lemma}\label{lemma: descriviamo composite knots}
    Let $K_1$, $K_2$, and $K_\#$ be knots as above. The set $T(K_\#,\partial) \subset \R{\partial \compl{K_\#}}$ equals the set in \eqref{eq: description T(K) composite knot}. Furthermore,
    \[
    H_{K_\#} = \left\{(\theta_\#,\psi_\#) \in T(K_\#,\partial) \middle| \text{either }  (\theta_1,\psi_1) \in H_{K_1} \text{ or } (\theta_2,\psi_2) \in H_{K_2} \right\}.
\]
\begin{proof}
    Let $T$ be the set in \eqref{eq: description T(K) composite knot}. Let $\rho \colon \fund{\compl{K_\#}} \to SU(2)$ be a representation such that $\restr{\rho}{\fund{\partial \compl{K_{\#}}}}$ has image in $\Ui$.
    We call $\rho_1$ and $\rho_2$ the restrictions of $\rho$ to $\fund{\compl{K_1}}$ and $\fund{\compl{K_2}}$. By definition, we have that
    \[
        \restr{\rho}{\fund{\partial \compl{K_\#}}} \in T(K_\#,\partial), \quad \rho_1 \in T(K_1,\partial), \quad \text{and} \quad \rho_2 \in T(K_2,\partial).
    \]
    The presentation in \eqref{eq: presentazione composite knot} implies that
    \[
    \rho_1(\lambda_1)\rho_2(\lambda_2)=\rho(\lambda_\#) \quad \text{and} \quad \rho_1(\mu_1)=\rho_2(\mu_1)=\rho(\mu_\#).
    \]
    This shows that $T(K_\#,\partial) \subseteq T$.
    
    Conversely, let $\eta \colon \fund{\partial \compl{K_\#}} \to \Ui$ be a representation in $T$. By definition, there exists a pair of $SU(2)$-representations $\rho_1$ and $\rho_2$ of $\fund{\compl{K_1}}$ and $\fund{\compl{K_2}}$ such that
    \[
    \rho_1(\lambda_1)\rho_2(\lambda_2)=\eta(\lambda_\#) \quad \text{and} \quad \rho_1(\mu_1)=\rho_2(\mu_1)=\eta(\mu_\#).
    \]
    The presentation \eqref{eq: presentazione composite knot} implies that $\eta \in T(K_\#,\partial)$ and this concludes that $T \subseteq T(K_\#,\partial)$.

    Let us focus on the second conclusion.
    Let $\eta \in T(K_\#,\partial)$ and let $\rho \colon \fund{ \compl{K_\#}} \to SU(2)$ be an extension.
    Again, the presentation in \eqref{eq: presentazione composite knot} implies that a representation $\fund{\compl{K_\#}} \to SU(2)$ has abelian image if and only if its restrictions to $\fund{\compl{K_1}}$ and $\fund{\compl{K_2}}$ both have abelian image.
    This implies that $\eta \in H_{K_\#}$ if and only if either $\restr{\rho}{\fund{\compl{K_1}}}$ or $\restr{\rho}{\fund{\compl{K_2}}}$ is irreducible.
    Equivalently, if and only either
    \[
        \restr{\rho}{\fund{\partial \compl{K_1}}} \in H_{K_1} \quad \text{or} \quad\restr{\rho}{\fund{\partial \compl{K_2}}} \in H_{K_2}.
    \]
\end{proof}
\end{lemma}

\begin{figure}[t]
        \centering
        \begin{tikzpicture}[>=latex,scale=2.4]
            \draw[->] (-.35,0) -- (2.35,0) node[right] {$\theta_K$};
            \foreach \x /\n in {0.2/$\frac{2\pi k}{r}$,1/$\frac{2\pi (k+1)}{r}$,} \draw[shift={(\x,0)}] (0pt,2pt) -- (0pt,-2pt) node[below] {\tiny \n};
            \foreach \x /\n in {1/$\frac{2\pi (k+1)}{r}$,1.8/$\frac{2\pi (k+2)}{r}$,} \draw[shift={(\x,2)}] (0pt,2pt) -- (0pt,-2pt) node[below] {\tiny \n};
            \draw[->] (0,-.35) -- (0,2.35) node[below right] {$\psi_K$};
            \foreach \y /\n in {1/$\pi$,2/$2\pi$}
            \draw[shift={(0,\y)}] (2pt,0pt) -- (-2pt,0pt) node[left] {\tiny \n};
            \node[below left] at (0,0) {\tiny $0$};
            \draw (0,2) -- (2,2);
            
              \draw [blue] (0.2,0)-- (1,2);
              \draw [blue] (1,0)-- (1.8,2);

            \draw [yellow] (0.2,0) --(1.8,2); 

            
              \draw [magenta, thick] plot[smooth] coordinates {(0.6,0) (1,0.5)(0.8,1) (1.3,1.5) (1.2,2)};
        \end{tikzpicture}
        \caption{The good arc $\Gamma$ in $R_{U(1)}(\partial E(K_1))$ drawn in magenta. The line $L_{r}^0$ is in blue and the in yellow are the
        fixed points of $\varphi^\ast$.}
        \label{figure: good arc for SU(2)-clean knot}
\end{figure}
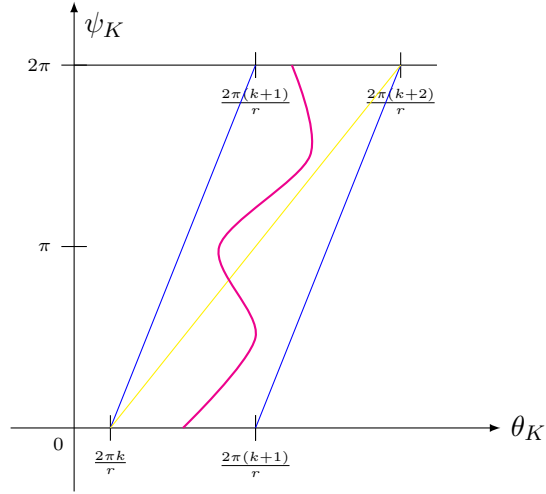

Let $\nicefrac{r}{s} \in \mathbb{Q}$ be a slope, if $r=2r'$, then with an abuse of notation we define $K(\nicefrac{r}{2s})$ as $K(\nicefrac{r'}{s})$.  

\begin{lemma}\label{lemma: K2(r/s) SU(2)ab then r/2s SU(2)ab}
    Let $K$ be a non-trivial knot. We denote by $K^2$ the double knot $K \# K$. If $K^2(\nicefrac{r}{s})$ is $SU(2)$-abelian, then $K(\nicefrac{r}{2s})$ is $SU(2)$-abelian as well.
    \begin{proof}
        Let us suppose that $K^2(\nicefrac{r}{s})$ is $SU(2)$-abelian. Thus,
        the manifold $K(\nicefrac{r}{s})$ is $SU(2)$-abelian by Lemma \ref{lemma: K2 > K1, se K1 ha una chiurgia irr allora K2 la ha}. Lemma \ref{lemma: descriviamo composite knots} implies that
        \[
            J \coloneqq \left\{ (\theta_K,2\psi_K) \middle| (\theta_K,\psi_K) \in H_K\right\} \subset \R{\partial\compl{K^2}}
        \]
        is a subset of $H_{K^2} \subset T(K^2,\partial)$. Hence, since we supposed that $K^2(\nicefrac{r}{s})$ is $SU(2)$-abelian, then $J \cap L_{\nicefrac{r}{s}}^0$ is empty by Remark \ref{rmk: surgery and the line}. In particular, 
        \[
             -\frac{r}{s}\theta_K + \frac{2\pi j}{q} \neq 2 \psi_K,
        \]
        with $(\theta_K,\psi_K) \in H_K$, $\theta_K \in [0,2\pi]$, and $j \in \left\{ 1, \cdots, |q|\right\}$. Therefore,
        \[
             \left(\theta_K,-\frac{r}{2s}\theta_K + \frac{\pi j}{q} \right) \neq (\theta_K,\psi_K) \in H_K,
        \]
         with $(\theta_K,\psi_K) \in H_K$, $\theta_K \in [0,2\pi]$, and $j \in \left\{ 1, \cdots, |q|\right\}$. As we said in Remark \ref{rmk: surgery and the line}, this implies that $L_{\nicefrac{r}{2s}}^0$ does not intersect $H_K$. Hence, $K(\nicefrac{r}{2s})$ is $SU(2)$-abelian.
    \end{proof}
\end{lemma}

\begin{lemma}\label{lemma: r/s and r/2s then |s|<=2}
    Let $K$ be a non-trivial knot and $\nicefrac{r}{s} \in \mathbb{Q}$ such that $K(\nicefrac{r}{s})$ and $K(\nicefrac{r}{2s})$ are $SU(2)$-abelian manifolds.
    If $r$ is odd, then $|s|= 1$. If $r$ is even, then $|s| \in \{1,3\}$.
    \begin{proof}
        If $r$ is odd, then by \cite[Corollary 1.10]{SU2cyclicSurgeryFormula}, we obtain that
        \[
        2\geo{\nicefrac{r}{s}}{\nicefrac{r}{2s}}=2|2rs-rs|=2|rs| \le 2 |r|.
        \]
        Therefore, $|s| = 1$.
        If $r=2r'$ is even, then by \cite[Theorem 1.9]{SU2cyclicSurgeryFormula} we obtain that
        \[
            \geo{\nicefrac{2r'}{s}}{\nicefrac{r'}{s}}=|2r's-r's|= |r's| \le |2r'|+|r'|=3|r'|,
        \]
        which implies that $|s| \le 3$. Since $\gcd(r,s)=1$, we conclude that $|s|\neq 2$.
    \end{proof}
\end{lemma}

 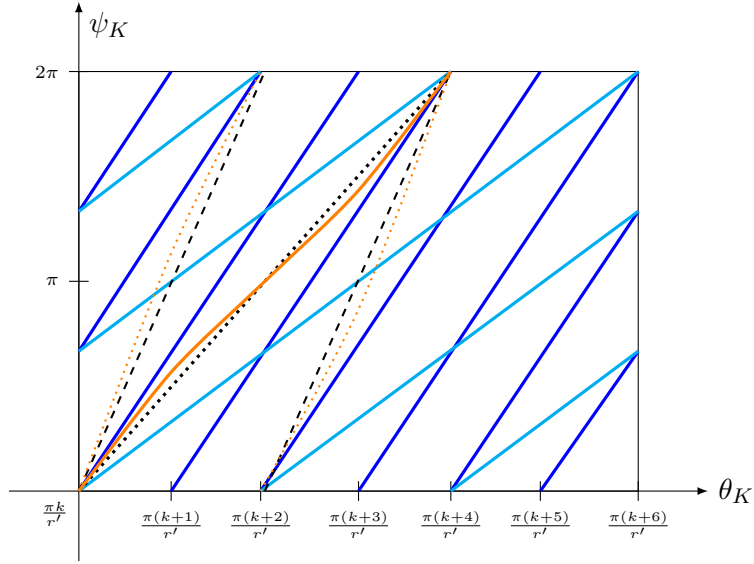
\begin{figure}[t]
    \centering
        \begin{tikzpicture}[>=latex,scale=1.85]
            \draw[->] (-.5,0) -- (4.5,0) node[right] {$\theta_K$};
            \foreach \x /\n in {0.66/$\frac{\pi (k+1)}{r'}$,1.3/$\frac{\pi (k+2)}{r'}$,2/$\frac{\pi (k+3)}{r'}$,2.66/$\frac{\pi (k+4)}{r'}$,3.3/$\frac{\pi (k+5)}{r'}$,4/$\frac{\pi (k+6)}{r'}$} \draw[shift={(\x,0)}] (0pt,2pt) -- (0pt,-2pt) node[below] {\tiny \n};
            \draw[->] (0,-.5) -- (0,3.5) node[below right] {$\psi_K$};
            \foreach \y /\n in {1.5/$\pi$,3/$2\pi$}
            \draw[shift={(0,\y)}] (2pt,0pt) -- (-2pt,0pt) node[left] {\tiny \n};
            \node[below left] at (0,0) {\tiny $\frac{\pi k}{r'}$};
            \draw (0,0) rectangle (4,3);


            \draw [blue, very thick] (0,0)-- (2,3);
            \draw [blue, very thick] (0.66,0)-- (2.66,3);
            \draw [blue, very thick] (1.3,0)-- (3.3,3);
            \draw [blue, very thick] (2,0)-- (4,3);
            \draw [blue, very thick] (2.66,0)-- (4,2);
            \draw [blue, very thick] (0,2)-- (0.66,3);
            \draw [blue, very thick] (3.3,0)-- (4,1);
            \draw [blue, very thick] (0,1)-- (1.3,3);


             \draw [cyan, very thick] (0,0)-- (4,3);
             \draw [cyan, very thick] (1.3,0)-- (4,2);
             \draw [cyan, very thick] (0,2)-- (1.3,3);
             \draw [cyan, very thick] (2.66,0)-- (4,1);
             \draw [cyan, very thick] (0,1)-- (2.66,3);

            \draw[dotted, very thick] (0,0)--(2.66,3);
            \draw[dashed,  thick] (0,0)--(1.33,3);
            \draw[dashed,  thick] (1.33,0)--(2.66,3);

            \draw [orange,very  thick] plot[smooth] coordinates {(0,0) (0.66,3/4+0.1)(1.33,6/4)(2,9/4-0.1) (2.66,3)};

             \draw [orange, thick, dotted] plot[smooth] coordinates {(0,0) (0.66,6/4+0.2)(1.3,3)};
             \draw [orange, thick, dotted] plot[smooth] coordinates {(1.3,0) (2,6/4-0.2) (2.66,3)};


        \end{tikzpicture}
        \caption{The lines $L_{\nicefrac{r}{3}}^0$ and $L_{\nicefrac{r'}{3}}^0$ in blue and cyan. The good arc $\Gamma$ is drawn in orange.}
        \label{Figure for L(r/3) and L(r/6)}
\end{figure}

\begin{repteo}{teo: double knot}
    Let $K$ be a non-trivial $SU(2)$-clean knot. Then $K^2=K\#K$ is not $SU(2)$-abelian.
    \begin{proof}
        Let $\nicefrac{r}{s} \in \mathbb{Q}$. We will prove that if $K^2(\nicefrac{r}{s})$ is $SU(2)$-abelian, then $\nicefrac{r}{s}=\nicefrac{1}{0}$.
        
        Let us suppose that $K^2(\nicefrac{r}{s})$ is $SU(2)$-abelian. According to Lemma \ref{lemma: K2(r/s) SU(2)ab then r/2s SU(2)ab}, $K(\nicefrac{r}{s})$ and $K(\nicefrac{r}{2s})$ are $SU(2)$-abelian.
        If $r \equiv_2 1$, then Lemma \ref{lemma: r/s and r/2s then |s|<=2} implies that $|s| =1$. If $r \equiv_2 0$, then Lemma \ref{lemma: r/s and r/2s then |s|<=2} implies that $|s| \in \{1,3\}$.

        If $|s|=1$, then the conclusion holds by Proposition \ref{prop: integer surgery on composite}. Let us suppose that $|s|=3$ and therefore $r \equiv_20$. We write $r=2r'$.
        The surgeries $K(\nicefrac{2r'}{s})$ and $K(\nicefrac{r'}{s})$ are both $SU(2)$-abelian by assumption. As we said in Remark \ref{rmk: surgery and the line}, this implies that
        \begin{align}
            \label{eq: una volta grid}
            H_K \cap
            \Lambda'= \emptyset \quad \text{where} \quad \Lambda' \coloneq
            \left( L^0_{\nicefrac{2r'}{3}} \cup L^0_{\nicefrac{r'}{3}} \right) \subset \R{\partial \compl{K}}.
        \end{align}

   We recall that by Definition \ref{defn: linee p/q}, the line $L_{\nicefrac{2r'}{3}}^0$ is made of segments connecting the points
     \[
        \left( \frac{\pi k}{r'},0\right) \in \R{\partial \compl{K}} \quad \text{and} \quad \left( \frac{\pi (k+3)}{r'},0\right) \in \R{\partial \compl{K}}
    \]
    for a
    $k \in \{0,\cdots,2r'-1\}$.
    Similarly, line $L_{\nicefrac{r'}{3}}^0$ is made of segments connecting the points
    \[
         \left( \frac{2\pi k'}{r'},0\right)  \in \R{\partial \compl{K}} \quad \text{and} \quad  \left( \frac{2\pi (k'+3)}{r'},0\right)\in \R{\partial \compl{K}},
   \]
    with $k' \in \{0,\cdots,r'-1\}$.
    Thus, the intersection $L_{\nicefrac{2r'}{3}}^0 \cap \{\psi_K=0\}$ is the set of points of the form $ \left( \frac{\pi k}{r'},0\right)$ and the intersection $L_{\nicefrac{r'}{3}}^0 \cap \{\psi_K=0\}$ is made of points with $k$ even.
    The grid $\Lambda'$ and lines $L_{\nicefrac{2r'}{3}}^0$ and $L_{\nicefrac{r'}{3}}^0$ are drawn in Figure \ref{Figure for L(r/3) and L(r/6)}.
    
    Let $\Gamma \subset H_K$ be a good arc of $K$. Such an arc exists as a consequence of Corollary \ref{cor: every walk in R(partial K) from P to Q}.
    Since $\Gamma \subset H_K$, the identity \eqref{eq: una volta grid} implies that $\Gamma \cap \Lambda' = \emptyset$. 
    Figure \ref{Figure for L(r/3) and L(r/6)} shows that there exists a $k \in \{0,\cdots,2r'-1\}$ such that the arc $\Gamma$ connects the points
    \[
        \left(\frac{\pi k}{r'},0\right) \in \R{\partial \compl{K}} \quad \text{and} \quad \left(\frac{\pi (k+4)}{r'},0\right) \in \R{\partial \compl{K}}.
    \]
    Therefore, $\Gamma$ is homotopic in $\R{\partial \compl{K}}$ relative $\partial \Gamma$ to a segment of slope $\nicefrac{2r'}{4}=\nicefrac{r'}{2}$. This segment is represented in Figure \ref{Figure for L(r/3) and L(r/6)} with the dotted segment.
    As an application of Lemma \ref{lemma: descriviamo composite knots},
    \[
        \Gamma^2 \coloneqq \left\{(\theta_\#,\psi_\#) \in \R{\partial \compl{K^2}} \middle| \theta_\#=\theta_K, \psi_\# =  2\psi_K, (\theta_K,\psi_K)\in \Gamma\right\} \subset H_{K^2}.
    \]
    The arc $\Gamma^2$ is represented in Figure \ref{Figure for L(r/3) and L(r/6)} with an orange dotted segment.
    Therefore, $\Gamma^2$ connects the points
    \[
    \left(\frac{\pi k}{r'},0\right) \in \R{\partial \compl{K^2}} \quad \text{and} \quad \left(\frac{\pi (k+4)}{r'},0\right) \in \R{\partial \compl{K^2}}
    \]
    and it is homotopic in $\R{\partial\compl{K^2}}$ relative $\partial \Gamma^2$ to a segment of slope $2(\nicefrac{r'}{2})=r'$. This segment is represented in Figure \ref{Figure for L(r/3) and L(r/6)} with the dashed segment.
    In particular, this segment intersects the line $L_{\nicefrac{2r'}{3}}^0=L_{\nicefrac{r}{3}}^0$ three times: two points are $\partial \Gamma^2$ and a third middle one.
    This implies that $\Gamma^2$ intersects $L_{\nicefrac{2r'}{3}}^0=L_{\nicefrac{r}{3}}^0$ and therefore that $H_{K^2} \cap L_{\nicefrac{2r'}{3}}^0 \neq \emptyset$.
    According to Remark \ref{rmk: surgery and the line}, this implies that $K^2(\nicefrac{r}{3}) = K^2(\nicefrac{r}{s})$ is not $SU(2)$-abelian.
    This is a contradiction since we assumed that $K^2(\nicefrac{r}{s})$ is $SU(2)$-abelian.

    We conclude that if $K^2(\nicefrac{r}{s})$ is $SU(2)$-abelian, then $\nicefrac{r}{s}=\nicefrac{1}{0}$.
\end{proof}
\end{repteo}

\section{Menagerie of non-SU(2)-abelian knots}

In this section we prove Corollary \ref{cor: composite torus knots are not SUa} and Corollary \ref{cor: knots up to 9 crosssings}.
We start by showing a lemma that will be the main tool in Corollary \ref{cor: composite torus knots are not SUa}.

\begin{lemma}\label{lemma: stepping stone}
    Let $T_{p_1,p_2} \subset S^3$ be a torus knot with $2 \le p_1 < p_2$.
    If $p_1=2$, then
    \[
        H_{T_{p_1,p_2}} = \left\{(\theta,\psi) \in \R{\partial \compl{T_{p_1,p_2}}} \middle| 2p_2 \theta -\psi = \pi\right\}_{\theta \in \left( \frac{\pi}{p_2}, \pi-\frac{\pi}{p_2}   \right) \cup \left( \pi+ \frac{\pi}{p_2}, 2\pi -\frac{\pi}{p_2}   \right)}.
    \]
    If $p_1 \ge 3$, then
    \[
        \left\{(\theta,\psi) \in \R{\partial \compl{T_{p_1,p_2}}} \middle| p_1p_2 \theta -\psi = 0\right\}_{\theta \in \left( \frac{2\pi}{p_1p_2}, \pi-\frac{2\pi}{p_1p_2}\right) \cup \left( \pi+ \frac{2\pi}{p_1p_2}, 2\pi -\frac{2\pi}{p_1p_2}   \right)} \subseteq H_{T_{p_1,p_2}}.
    \]
    \begin{proof}
        The claim follows from Lemmas $7.8$ and $7.13$ in \cite{Mino}.
    \end{proof}
\end{lemma}

\begin{repcor}{cor: composite torus knots are not SUa}
    Let $T_{p_1,p_2}$ and $T_{p_3,p_4}$ be two non-trivial torus knots. Then $T_{p_1,p_2} \# T_{p_3,p_4}$ is not $SU(2)$-abelian.
    \begin{proof}
        Let $K$ be the composite knot $T_{p_1,p_2} \# T_{p_3,p_4}$ and $\nicefrac{r}{s} \in \mathbb{Q} \cup \{\nicefrac{1}{0}\}$ a slope such that $K(\nicefrac{r}{s})$ is $SU(2)$-abelian. We are going to show that $\nicefrac{r}{s}=\nicefrac{1}{0}$.

        By Proposition \ref{prop: satellite P(J) ammette una chirugia SU(2)-abelian then}, if $K(\nicefrac{r}{s})$ is $SU(2)$-abelian, then $T_{p_1,p_2}(\nicefrac{r}{s})$ and $T_{p_3,p_4}(\nicefrac{r}{s})$ are both $SU(2)$-abelian. As an application of Theorem \ref{thm:iterated-cables}, either
        \[
            p_1p_2=p_3p_4 \quad \text{or} \quad |p_1p_2-p_3p_4|\le 2.
        \]
        Remark \ref{rmk mirror image} implies that a knot $K$ is not an $SU(2)$-abelian knot if and only if its mirror image $\overline{K}$ is not an $SU(2)$-abelian knot.
        We recall that $\overline{K_1 \# K_2} = \overline{K_1} \# \overline{K_2}$.
        Moreover, since $\overline{T_{p_1,p_2}} = T_{-p_1,p_2} = T_{p_1,-p_2}$ and $T_{p_1,p_2} = T_{-p_1,-p_2} = T_{p_2,p_1}$ (see \cite[page 55]{Rolfsen}), we may assume all $p_i$ are positive without loss of generality.

        Lemma \ref{lemma: stepping stone} implies that there exist $0 <x_1<x_2<x_3<x_4< 2\pi$ such that:
        \begin{itemize}
            \item there exists a $\psi \in (0,2\pi)$ such that $H_{T_{p_1,p_2}}$ contains a segment connecting the points of $\R{\partial \compl{T_{p_1,p_2}}}$ of coordinates $(x_1,\psi)$ and $(x_4,0)$;
            \item the segment of the previous point contains $(x_2,0) \in \R{\partial \compl{T_{p_1,p_2}}}$;
            \item $H_{T_{p_3,p_4}}$ contains a segment connecting the points of $\R{\partial \compl{T_{p_3,p_4}}}$ of coordinates $(x_1,\psi)$ and $(x_3,0)$.
        \end{itemize}
        These sets are represented in Figure \ref{figure: T1 proof T1T2 not SU(2)-abelian} and Figure \ref{figure: T2 proof T1T2 not SU(2)-abelian} respectively.
        These two subsets define a subset of $H_{T_{p_1,p_2}\#T_{p_3,p_4}}$ as explained in Lemma \ref{lemma: descriviamo composite knots}. This subset of $H_{T_{p_1,p_2}\#T_{p_3,p_4}}$ is shown in Figure \ref{figure: closed path composite torus knot}. In particular, Figure \ref{figure: closed path composite torus knot} shows that $H_{T_{p_1,p_2}\#T_{p_3,p_4}}$ contains a closed path $\Gamma$ which is homotopic to the circle $\{\psi_\#=0\}$. Therefore, the circle $L^0_{\nicefrac{r}{s}}$ has at least $|s|$ intersections with $\Gamma$. 
        Remark \ref{rmk: surgery and the line} implies that $K(\nicefrac{r}{s})$ is $SU(2)$-abelian if and only if $|s|=0$ and therefore $\nicefrac{r}{s}=\nicefrac{1}{0}$.
    \end{proof}
\end{repcor}

\begin{figure}[t]
    \centering
    \begin{subfigure}[b]{0.45\textwidth}
        \centering
         \begin{tikzpicture}[>=latex,scale=1.5]
            \draw[->] (-.5,0) -- (2.5,0) node[right] {$\psi_1$};
            \foreach \x /\n in {1/$\pi$,2/$2\pi$} \draw[shift={(\x,0)}] (0pt,2pt) -- (0pt,-2pt) node[below] {\tiny \n};
            \draw[->] (0,-.5) -- (0,2.5) node[below right] {$\theta_1$};
            \foreach \y /\n in {0.4/$x_1$,0.6/$x_2$,1.4/$x_3$,1.6/$x_4$}
            \draw[shift={(0,\y)}] (2pt,0pt) -- (-2pt,0pt) node[left] {\tiny \n};
            \node[below left] at (0,0) {\tiny $0$};
             \draw (0,0) rectangle (2,2);
             \draw (1,0) -- (1,2); 
             \draw[dotted] (0,0.4)--(2,0.4);
             \draw[dotted] (0,0.6)--(2,0.6);
             \draw[dotted] (0,1.4)--(2,1.4);
             \draw[dotted] (0,1.6)--(2,1.6);

             \draw [orange,  ultra thick] (0,0)--(0,2);
             \draw [orange,  ultra thick] (2,0)--(2,2);

             \draw[magenta,  ultra thick] (0,1.6)--(2,0.6);
             \draw[magenta,  ultra thick](0,0.6)--(0.4,0.4);
        \end{tikzpicture}
        \caption{A subset of $H_{T_{p_1,p_2}} \subset \R{\partial \compl{T_{p_1,p_2}}}$.}
        \label{figure: T1 proof T1T2 not SU(2)-abelian}
    \end{subfigure}
    \quad
     \begin{subfigure}[b]{0.45\textwidth}
     \centering
        \begin{tikzpicture}[>=latex,scale=1.5]
            \draw[->] (-.5,0) -- (2.5,0) node[right] {$\psi_2$};
            \foreach \x /\n in {1/$\pi$,2/$2\pi$} \draw[shift={(\x,0)}] (0pt,2pt) -- (0pt,-2pt) node[below] {\tiny \n};
            \draw[->] (0,-.5) -- (0,2.5) node[below right] {$\theta_2$};
            \foreach \y /\n in {0.4/$x_1$,0.6/$x_2$,1.4/$x_3$,1.6/$x_4$}
            \draw[shift={(0,\y)}] (2pt,0pt) -- (-2pt,0pt) node[left] {\tiny \n};
            \node[below left] at (0,0) {\tiny $0$};
             \draw (0,0) rectangle (2,2);
             \draw (1,0) -- (1,2); 

            \draw [orange,  ultra thick] (0,0)--(0,2);
             \draw [orange,  ultra thick] (2,0)--(2,2);
             \draw[dotted] (0,0.4)--(2,0.4);
             \draw[dotted] (0,0.6)--(2,0.6);
             \draw[dotted] (0,1.4)--(2,1.4);
             \draw[dotted] (0,1.6)--(2,1.6);
             \draw[magenta,  ultra thick] (0,1.4)--(2,0.4);
        \end{tikzpicture}
        \caption{A subset of $H_{T_{p_3,p_4}} \subset \R{\partial \compl{T_{p_3,p_4}}}$.}
        \label{figure: T2 proof T1T2 not SU(2)-abelian}
    \end{subfigure}
    \\[5pt]
    \begin{subfigure}[b]{0.45\textwidth}
     \centering
        \begin{tikzpicture}[>=latex,scale=1.5]
            \draw[->] (-.5,0) -- (2.5,0) node[right] {$\psi_\#$};
            \foreach \x /\n in {1/$\pi$,2/$2\pi$} \draw[shift={(\x,0)}] (0pt,2pt) -- (0pt,-2pt) node[below] {\tiny \n};
            \draw[->] (0,-.5) -- (0,2.5) node[below right] {$\theta_\#$};
            \foreach \y /\n in {0.4/$x_1$,0.6/$x_2$,1.4/$x_3$,1.6/$x_4$}
            \draw[shift={(0,\y)}] (2pt,0pt) -- (-2pt,0pt) node[left] {\tiny \n};
            \node[below left] at (0,0) {\tiny $0$};
             \draw (0,0) rectangle (2,2);
             \draw (1,0) -- (1,2); 

            \draw [orange,  ultra thick] (0,0)--(0,2);
             \draw [orange,  ultra thick] (2,0)--(2,2);
             \draw[dotted] (0,0.4)--(2,0.4);
             \draw[dotted] (0,0.6)--(2,0.6);
             \draw[dotted] (0,1.4)--(2,1.4);
             \draw[dotted] (0,1.6)--(2,1.6);
             \draw[magenta,  ultra thick] (0,1.4)--(2,0.4);
             \draw[magenta,  ultra thick] (0,1.6)--(2,0.6);
             \draw[purple,  ultra thick] (0.4,1.4)--(2,1);
             \draw[purple,  ultra thick] (0,1)--(2,0.5);
             \draw[purple,  ultra thick] (0,0.5)--(0.4,0.4);
             \draw[magenta,  ultra thick](0,0.6)--(0.4,0.4);
              \draw[purple, dotted, thick] (2,0.5)--(2.4,0.4);
              \draw[magenta, dotted, thick] (2,0.6)--(2.4,0.4);
               \draw[line width=4pt, yellow, decorate, opacity=0.4](0.4,1.4) -- (2,1);
               \draw[line width=4pt, yellow, decorate, opacity=0.4](0,1) -- (2,0.5);
               \draw[line width=4pt, yellow, decorate, opacity=0.4](0,0.5) -- (0.4,0.4);
               \draw[line width=4pt, yellow, decorate, opacity=0.4](0.4,1.4) -- (2,0.6);
               \draw[line width=4pt, yellow, decorate, opacity=0.4](0,0.6) -- (0.4,0.4);
        \end{tikzpicture}
        \caption{A subset of $H_{T_{p_1,p_2} \#T_{p_3,p_4}}$.}
     \label{figure: closed path composite torus knot}
    \end{subfigure}
     \caption{In the pictures the coordinates are as usual.}
     \label{figure: proof composite torus knots}
\end{figure}
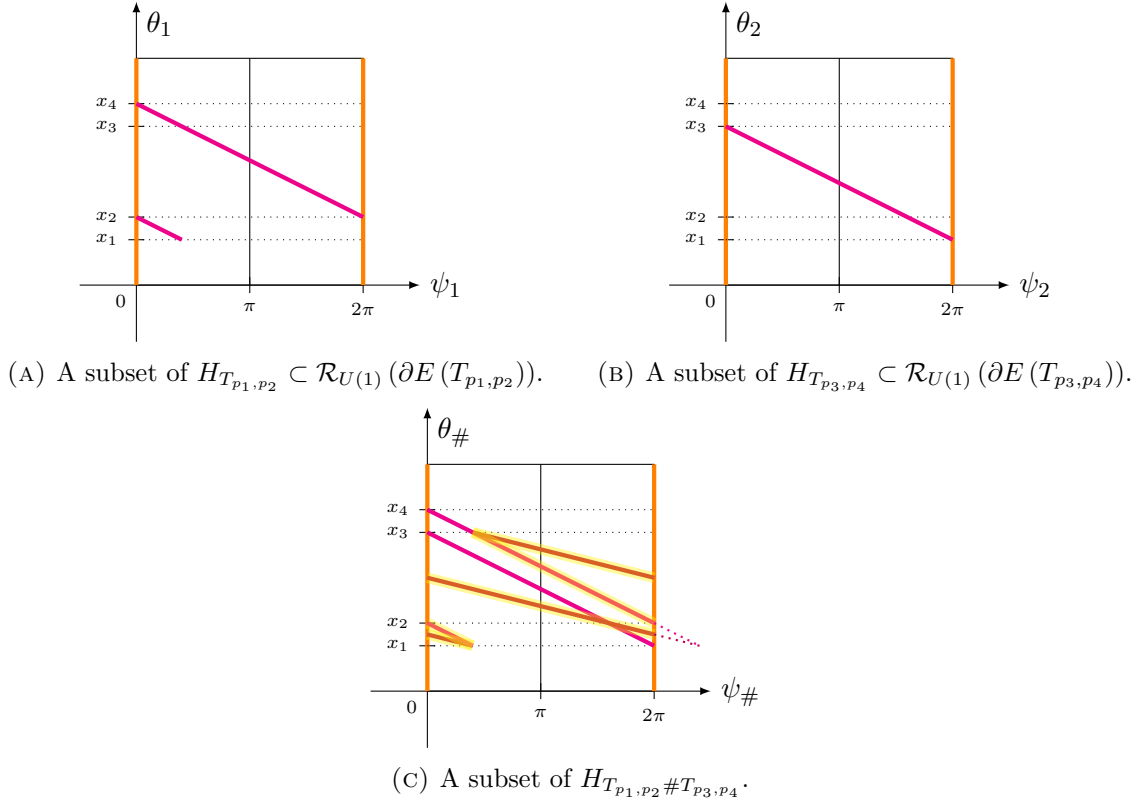

\begin{exmp}\label{exmp: un altro po di nodi}
Let $K$ be in the set 
    \[
    \left\{8_{18}, 9_{37}, 9_{40}, 10_{58},10_{60},10_{74},10_{122},10_{136},10_{138}\right\}
\]
It is proven in \cite{ApartialOrderInTheKnotTableII} that there exists a $2$-bridge non-toroidal knot $K_0 \subset S^3$ such that $K \ge K_0$.
According to \cite[Theorem 1.1]{bowden2007winding}, the knot $K_0$ is not $SU(2)$-abelian.
By Corollary \ref{cor: K2 > K1, if K1 is not SU(2)abelian, then K2 is not}, $K$ is not $SU(2)$-abelian.
\end{exmp}

\begin{lemma}\label{lemma: trefoil 5meridian are not SU(2)abelian}
    Let $K$ be the trefoil. If $r \ge3$, then the group
    \[
    \frac{\fund{\compl{K}}}{\normalsubgroup{\mu^r}}
    \]
    admits an irreducible $SU(2)$-representation.
\begin{proof}
    We consider the torus $\R{\partial \compl{K}}$ with $(\theta,\psi)$ coordinates as in \eqref{eq: coordinates of R(partial E(K))}.
    It is well known that
    \[
        H_{K} = \left\{(\theta,\psi) \in L_{6}^0 \middle| \theta \in \left( \frac{\pi}{6},\frac{5\pi}{6} \right) \cup  \left( \frac{7\pi}{6},\frac{11\pi}{6} \right) \right\} \subset \R{\partial \compl{K}}.
    \]
    Here the line $L_6^0$ is as in Definition \ref{defn: linee p/q}.
    See Figure \ref{figure: trefoild mu5}.
    It is easy to see that, since $r \ge 3$, there exists a $k \in \{0,\cdots r-1\}$ such that
    \[
        \frac{2\pi k}{r} \in \left( \frac{\pi}{6},\frac{5\pi}{6} \right) \cup  \left( \frac{7\pi}{6},\frac{11\pi}{6} \right).
    \]
    This implies that there exists an irreducible representation $\rho \fund{\compl{K}}\to SU(2)$ such that
    \[
        \rho(\mu) = \begin{bmatrix}
            e^{\frac{2\pi i k}{r}} & 0 \\ 0 & e^{-\frac{2\pi i k}{r}} 
        \end{bmatrix}
        \quad, \text{and hence} \quad \rho(\mu)^r=1.
    \]
    This representation factors through the group $\fund{\compl{K}}/\normalsubgroup{\mu^r}$. Thus, the latter is not an $SU(2)$-abelian group. 
    \end{proof}
\end{lemma}
\begin{lemma}\label{lemma: figure eight 5 7 meridian are not SU(2)abelian}
    Let $K$ be the figure eight knot. If $r \ge 3$, then the group
    \[
    \frac{\fund{\compl{K}}}{\normalsubgroup{\mu^r}},
    \]
    admits an irreducible $SU(2)$-representation, where $\mu \in \fund{\compl{K}}$ is the homotopy class of the meridian.
    \begin{proof}
    The proof is similar to the one of Lemma \ref{lemma: trefoil 5meridian are not SU(2)abelian}.
Let $K=4_1$ be the figure eight knot.
Let
\[
    \rho \colon \fund{\compl{K}} \to SU(2)
\]
be an irreducible $SU(2)$-representation.
By \cite[Proposition 5.4]{DescriptionLongitude} there exists a $s \in [\nicefrac{\pi}{3},\nicefrac{2\pi}{3}] \cup [\nicefrac{4\pi}{3},\nicefrac{5\pi}{3}] $ such that
\[
\rho(\mu) = \begin{bmatrix}
    e^{ i s}&  0 \\ 0 &e^{- i s}
\end{bmatrix},
\]
up to conjugation. See Figure \ref{figure: figure eight knot}. It is easy to see that, since $r \ge 3$, there exists a $k \in \{0,\cdots r-1\}$ such that
    \[
        \frac{2\pi k}{r} \in \left[\frac{\pi}{3},\frac{2\pi}{3}\right] \cup \left[\frac{4\pi}{3},\frac{5\pi}{3}\right].
    \]
    This implies that there exists an irreducible representation $\rho \colon \fund{\compl{K}}\to SU(2)$ such that
    \[
        \rho(\mu) = \begin{bmatrix}
            e^{\frac{2\pi i k}{r}} & 0 \\ 0 & e^{-\frac{2\pi i k}{r}} 
        \end{bmatrix}, \quad \text{and hence} \quad \rho(\mu)^r=1.
    \]
    This representation factors through the group $\fund{\compl{K}}/\normalsubgroup{\mu^r}$. Thus, the latter is not an $SU(2)$-abelian group. 
\end{proof}
\end{lemma}

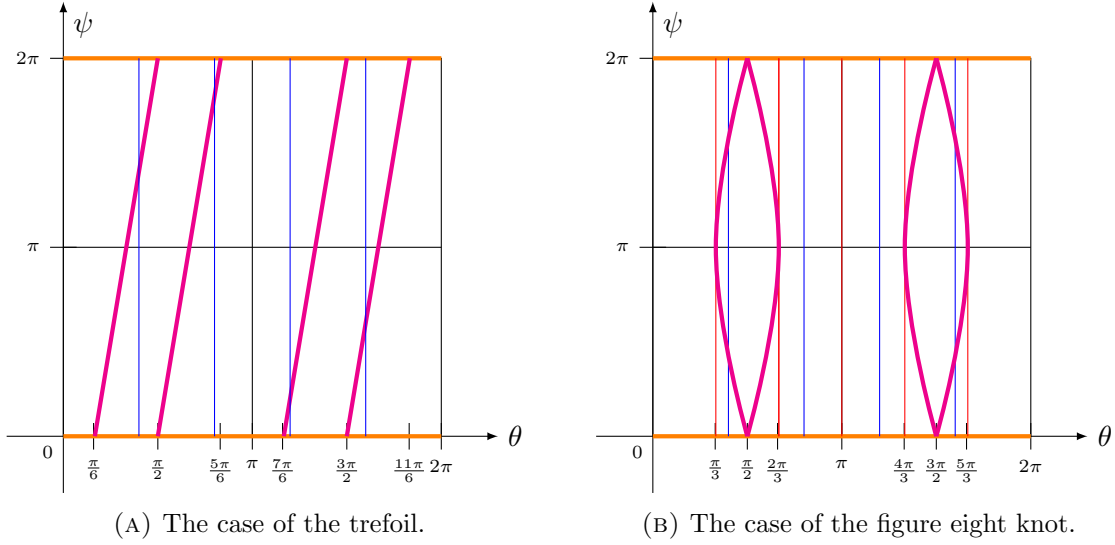
\begin{figure}[t]
    \centering
    \begin{subfigure}[b]{0.45\textwidth}
    \centering
        \begin{tikzpicture}[>=latex,scale=2.5]
            \draw[->] (-.3,0) -- (2.3,0) node[right] {$\theta$};
            \foreach \x /\n in {1/$\pi$,2/$2\pi$, 0.16/$\frac{\pi}{6}$, 0.5/$\frac{\pi}{2}$, 0.83/$\frac{5\pi}{6}$, 1.16/$\frac{7\pi}{6}$, 1.5/$\frac{3\pi}{2}$, 1.83/$\frac{11\pi}{6}$} \draw[shift={(\x,0)}] (0pt,2pt) -- (0pt,-2pt) node[below] {\tiny \n};
            \draw[->] (0,-.3) -- (0,2.3) node[below right] {$\psi$};
            \foreach \y /\n in {1/$\pi$,2/$2\pi$}
            \draw[shift={(0,\y)}] (2pt,0pt) -- (-2pt,0pt) node[left] {\tiny \n};
            \node[below left] at (0,0) {\tiny $0$};
             \draw (0,0) rectangle (2,2);
             \draw (1,0) -- (1,2); 
             \draw (0,1) -- (2,1);

             \draw [orange,  ultra thick] (0,0)--(2,0);
             \draw [orange,  ultra thick] (0,2)--(2,2);

             \draw [magenta, ultra thick] (1/6,0)--(1/2,2);
             \draw [magenta, ultra thick] (1/2,0)--(5/6,2);
             \draw [magenta, ultra thick] (1+1/6,0)--(1+1/2,2);
             \draw [magenta, ultra thick] (1+1/2,0)--(1+5/6,2);
             
             \draw [blue] (2/5,0)--(2/5,2);
             \draw [blue] (4/5,0)--(4/5,2);
             \draw [blue] (6/5,0)--(6/5,2);
             \draw [blue] (8/5,0)--(8/5,2);
        \end{tikzpicture}
\caption{The case of the trefoil.}
\label{figure: trefoild mu5}
\end{subfigure} \quad
\begin{subfigure}[b]{0.45\textwidth}
            \centering
        \begin{tikzpicture}[>=latex,scale=2.5]
            \draw[->] (-.3,0) -- (2.3,0) node[right] {$\theta$};
            \foreach \x /\n in {1/$\pi$,2/$2\pi$, 0.33/$\frac{\pi}{3}$, 0.5/$\frac{\pi}{2}$, 0.66/$\frac{2\pi}{3}$, 1.33/$\frac{4\pi}{3}$, 1.5/$\frac{3\pi}{2}$, 1.66/$\frac{5\pi}{3}$} \draw[shift={(\x,0)}] (0pt,2pt) -- (0pt,-2pt) node[below] {\tiny \n};
            \draw[->] (0,-.3) -- (0,2.3) node[below right] {$\psi$};
            \foreach \y /\n in {1/$\pi$,2/$2\pi$}
            \draw[shift={(0,\y)}] (2pt,0pt) -- (-2pt,0pt) node[left] {\tiny \n};
            \node[below left] at (0,0) {\tiny $0$};
             \draw (0,0) rectangle (2,2);
             \draw (1,0) -- (1,2); 
             \draw (0,1) -- (2,1);

             \draw [orange,  ultra thick] (0,0)--(2,0);
             \draw [orange,  ultra thick] (0,2)--(2,2);
             
             \draw [blue] (2/5,0)--(2/5,2);
             \draw [blue] (4/5,0)--(4/5,2);
             \draw [blue] (6/5,0)--(6/5,2);
             \draw [blue] (8/5,0)--(8/5,2);

              \draw [red] (2/6,0)--(2/6,2);
             \draw [red] (4/6,0)--(4/6,2);
              \draw [red] (6/6,0)--(6/6,2);
              \draw [red] (8/6,0)--(8/6,2);
                 \draw [red] (10/6,0)--(10/6,2);
                  \draw [red] (4/6,0)--(4/6,2);
             \draw [magenta, ultra thick] plot [smooth, tension=0.75] coordinates { (1/2,0) (2/3,1) (1/2,2)};
            \draw [magenta, ultra thick] plot [smooth, tension=0.75] coordinates { (1/2,0) (1/3,1) (1/2,2)};

            \draw [magenta, ultra thick] plot [smooth, tension=0.75] coordinates { (1+1/2,0) (1+2/3,1) (1+1/2,2)};
            \draw [magenta, ultra thick] plot [smooth, tension=0.75] coordinates { (1+1/2,0) (1+1/3,1) (1+1/2,2)};
        \end{tikzpicture}
        \caption{The case of the figure eight knot.}
        \label{figure: figure eight knot}
\end{subfigure}
\caption{The space $T(K,\partial) \subset \R{\partial \compl{K}}$ for the trefoil $3_1$ and the figure eight knot $4_1$, the sets $A_K$ and $H_{K}$ are in orange and magenta respectively. The set $\{\theta^5\equiv_{2\pi} 0\}$ is in blue and the set $\{\theta^6\equiv_{2\pi} 0$\} is in red.}
\end{figure}

\begin{lemma}\label{lemma: r=3,4 then K(r/s) not SU(2)abelian}
    Let $r \in \{1,2,3,4\}$ be an integer and $K$ a non-trivial knot. The manifold $K(\nicefrac{r}{s})$ is not $SU(2)$-abelian.
    \begin{proof}
        Since $r \le 4$, we obtain that $|\nicefrac{r}{s}| \in [0,4]$.
        If $|\nicefrac{r}{s}| \le 2$, then the conclusion holds by \cite[Theorem 1]{KronheimerMrowkaDehnSurgeryFundamental}.
        If $|\nicefrac{r}{s}|\in (2,3)$, since $r$ is a prime power, the conclusion holds by \cite[Theorem 1.8]{InstantonsLspaceSurgeries}.
        If $|\nicefrac{r}{s}|=3$, then \cite[Theorem 1.2]{baldwin2021small} implies that $K(3)$ is not $SU(2)$-abelian.
        If $|\nicefrac{r}{s}| \in (3,4]$, then $r=4$. The conclusion is implied by \cite[Theorem 1.4]{baldwin2021small}.
    \end{proof}
\end{lemma}

\begin{cor}\label{cor: un po di nodi}
    Let $K$ be a knot in the set
    \[
       \{8_{10},8_{20},9_{24},10_{59},10_{62},10_{63},10_{77},10_{82},10_{87},10_{99},10_{137},10_{140},10_{143}\}, 
    \]
    then $K$ is not $SU(2)$-abelian.
    \begin{proof}
        Let $K_0$ be either the trefoil $3_1$ or the figure eight knot $4_1$.
        In \cite{ApartialOrderInTheKnotTableII} it is proven that there exists a surjective epimorphism 
        \[
            \psi \colon \fund{\compl{K}} \twoheadrightarrow \fund{\compl{K_0}}
        \]
        such that $\psi(\mu_K)=\mu_{K_0}$ and $\psi(\lambda_K)= 1 \in \fund{\compl{K_0}}$. For a given surgery coefficient $\nicefrac{r}{s}\in\mathbb{Q}$, the map $\psi$ induces a surjective homomorphism
        \[
            \fund{K(\nicefrac{r}{s})} \twoheadrightarrow \frac{\fund{K_0}}{\normalsubgroup{\mu_{K_0}^r}}.
        \]
        
        If $|r| \le  4$, then by Lemma \ref{lemma: r=3,4 then K(r/s) not SU(2)abelian}, the manifold $K(\nicefrac{r}{s})$ is not $SU(2)$-abelian.

        If $|r| \ge 5$, then $K(\nicefrac{r}{s})$ is not $SU(2)$-abelian according to Lemma \ref{lemma: trefoil 5meridian are not SU(2)abelian} and Lemma \ref{lemma: figure eight 5 7 meridian are not SU(2)abelian}.
    \end{proof}
\end{cor}

\begin{cor}\label{corollary nodi}
    The following knots are not $SU(2)$-abelian knots
    \begin{multline*}
\{8_{10}, 8_{18}, 8_{20}, 9_{24}, 9_{37}, 9_{40}, 10_{58}, 10_{59}, 10_{60}, 10_{62},
10_{63}, 10_{74}, 10_{77},\\ 10_{82}, 10_{87}, 10_{99}, 10_{122}, 10_{136}, 10_{137}, 10_{138}, 10_{140}, 10_{143}
\}.
\end{multline*}
\begin{proof}
    The conclusion holds by Example \ref{exmp: un altro po di nodi} and Corollary \ref{cor: un po di nodi}.
\end{proof}
\end{cor}

\begin{defn}[{\cite[Defintion 1.1]{SU2AbundantLargeFormula}}]
    \label{defn: su2 abundant}
A knot $K \subset S^3$ is called $SU(2)$-abundant if the
following two conditions hold:
\begin{itemize}
    \item For all but finitely many $r \in \mathbb{Q}$, the manifold $K(r)$ is not $SU(2)$-abelian,ù
    \item For any $r=\nicefrac{p}{q} \in \mathbb{Q}$ such that $K(r)$ is $SU(2)$-abelian,
    there is some $p$-th root of unity $\zeta \in\mathbb{C}$ such that $\Delta_K(\zeta^2)=0$.
\end{itemize}
\end{defn}

\begin{rmk}
\label{rmk:su2-abundant-complement}
There are exactly $84$ prime knots with crossing number at most $9$.
Among them, only the following $7$ knots have Alexander polynomial with all coefficients in $\{-1,0,1\}$:
\[
3_1,\quad 5_1,\quad 7_1,\quad 8_{19},\quad 9_1,\quad 9_{47},\quad 9_{49}.
\]
Consequently, by \cite[Theorem 1.3]{SU2AbundantLargeFormula},
every prime knot with crossing number at most $9$ that is not in the above
list is $SU(2)$-abundant.

Among the $7$ exceptional knots, the knots
\[
3_1=T(2,3),\quad 5_1=T(2,5),\quad 7_1=T(2,7),\quad 8_{19}=T(3,4),\quad 9_1=T(2,9)
\]
are torus knots, whereas $9_{47}$ and $9_{49}$ are non-torus knots, each of bridge number $3$.
\end{rmk}

\begin{rmk}
\label{rmk:non2bridge-roots}
For an integer $n \ge 2$, we denote by $\Phi_n(t) \in \mathbb{Z}[t]$ the
$n$-th cyclotomic polynomial. This is an irreducible
polynomial over $\mathbb{Q}$ whose roots are
the primitive $n$-th root of unity.
In particular $\Phi_6(t) = t^2-t+1$, its roots are the primitive $6$-th roots of unity $e^{\pm i\pi/3}$.
We record into Table \ref{table: alexander polynomial} the Alexander polynomial factored over $\mathbb{Q}$ of some knots.
In particular, for a knot $K \subset S^3$ its Alexander polynomial $\Delta_K(t)$
has a primitive root of unity as a root of order $p$ if and only if $\Phi_{p}(t)$ appers in the decomposition of
$\Delta_K(t)$.
\end{rmk} 

\begin{table}[htbp]
    \centering
    \small
    \begin{tabular}{c|l|c}
$K$ & $\Delta_K(t)$, factored over $\mathbb{Q}$ & root of unity? \\
\hline
$8_5$   & $\Phi_6(t)(t^4-2t^3+t^2-2t+1)$ & Yes \\
$8_{15}$  & $\Phi_6(t)(3t^2-5t+3)$ & Yes  \\
$8_{16}$  & $t^6-4t^5+8t^4-9t^3+8t^2-4t+1$ & No \\
$8_{17}$  & $t^6-4t^5+8t^4-11t^3+8t^2-4t+1$ & No \\
$8_{21}$  & $(t^2-3t+1)\Phi_6(t)$ & Yes  \\
$9_{16}$  & $\Phi_6(t)(2t^4-3t^3+3t^2-3t+2)$ & Yes \\
$9_{22}$  & $t^6-5t^5+10t^4-11t^3+10t^2-5t+1$  & No \\
$9_{25}$  & $3t^4-12t^3+17t^2-12t+3$  & No \\
$9_{28}$  & $\Phi_6(t)(t^4-4t^3+7t^2-4t+1)$ & Yes  \\
$9_{29}$  & $\Phi_6(t)(t^4-4t^3+7t^2-4t+1)$ & Yes  \\
$9_{30}$  & $t^6-5t^5+12t^4-17t^3+12t^2-5t+1$ & No \\
$9_{32}$  & $t^6-6t^5+14t^4-17t^3+14t^2-6t+1$  & No \\
$9_{33}$  & $t^6-6t^5+14t^4-19t^3+14t^2-6t+1$  & No \\
$9_{34}$  & $t^6-6t^5+16t^4-23t^3+16t^2-6t+1$  & No \\
$9_{35}$  & $7t^2-13t+7$  & No \\
$9_{36}$  & $t^6-5t^5+8t^4-9t^3+8t^2-5t+1$  & No \\
$9_{38}$  & $\Phi_6(t)(5t^2-9t+5)$ & Yes  \\
$9_{39}$  & $(t^2-3t+1)(3t^2-5t+3)$ & No \\
$9_{41}$  & $(t^2-3t+3)(3t^2-3t+1)$ & No \\
$9_{42}$  & $t^4-2t^3+t^2-2t+1$  & No \\
$9_{43}$  & $t^6-3t^5+2t^4-t^3+2t^2-3t+1$  & No \\
$9_{44}$  & $t^4-4t^3+7t^2-4t+1$  & No \\
$9_{45}$  & $t^4-6t^3+9t^2-6t+1$  & No \\
$9_{46}$  & $(t-2)(2t-1)$ & No \\
$9_{48}$  & $t^4-7t^3+11t^2-7t+1$  & No \\
\end{tabular}
\caption{Decomposition over $\mathbb{Q}$ into irreducible factors of the Alexander polynomial of some knots.}
\label{table: alexander polynomial}
\end{table}

\begin{cor}
\label{cor:not-su2-abelian}
The following knots are not $SU(2)$-abelian:
\[
8_{16}, \, 8_{17},\, 9_{22},\, 9_{25},\, 9_{30},\, 9_{32},\, 9_{33},\, 9_{34},\, 9_{35},\, 9_{36},
\, 9_{39},\, 9_{41},\, 9_{42},\, 9_{43},\, 9_{44},\, 9_{45},\, 9_{46},\, \text{and } 9_{48}.
\]
\begin{proof}
By Remark \ref{rmk:su2-abundant-complement}, each
knot in the list is $SU(2)$-abundant.
By Remark \ref{rmk:non2bridge-roots}, each of these knots has $\Delta_K(t)$ with no root of unity among its roots.
Now fix such a $K$ in the list and $\nicefrac{p}{q}\in\mathbb{Q}\setminus\{0\}$.
Suppose for contradiction that $K(\nicefrac{p}{q})$ is $SU(2)$-abelian.
By Definition \ref{defn: su2 abundant} there must exist a $p$-th root of unity $\zeta$ with
$\Delta_K(\zeta^2)=0$. But $\zeta^2$ is itself a root of unity,
contradicting the fact that $\Delta_K(t)$ has no root of unity as a root.
Hence $K(\nicefrac{p}{q})$ admits an irreducible $SU(2)$ representation,
for every $\nicefrac{p}{q}\in\mathbb{Q}\setminus\{0\}$.
Since the $0$-surgery is never $SU(2)$-abelian when $K$ is non-trivial by \cite[Theorem 1.2]{KronheimerMrowkaDehnSurgeryFundamental},
we conclude that if $K$ is in the list, then it does not admit a non-trivial $SU(2)$-abelian surgery.
\end{proof}
\end{cor}

The group $SU(2)$ can be viewed either as the classical matrix group of $2\times 2$
unitary matrices of determinant one, or equivalently as the group of unit quaternions.
In the next definition we shall use the second description.
Further details may be found in \cite{saveliev}.

\begin{defn}
The binary tetrahedral group $2T \subset SU(2)$
is the subgroup generated by the unit quaternions
\[
s = \tfrac{1}{2}(1+i+j+k) \in SU(2), \qquad t = \tfrac{1}{2}(1+i+j-k) \in SU(2).
\]
\end{defn}
One can notice that $s^3=t^3=(st)^2=-1$.
Equivalently,
\[
2T=\{\pm1,\pm i,\pm j,\pm k\} \cup \{\tfrac{1}{2}(\pm1\pm i\pm j\pm k)\}.
\] 
This is a finite non-abelian subgroup of $SU(2)$ of order $24$.

\begin{table}[htbp]
\centering
\small
\begin{tabular}{c|L{4.6cm}|L{1.8cm}|L{5.6cm}}
$K$ & $\pi_1(S^3\setminus K)$ & $\mu$ & $\lambda$ \\
\hline
$8_5$ &
$\langle a,b,c \mid{}$\seqsplit{aacAABAcaab}$,\ $\seqsplit{abCBacb}$\rangle$ &
\seqsplit{ABCA} &
\seqsplit{acAACBCabaacccccccccB} \\[4pt]


$8_{15}$ &
$\langle a,b,c \mid{}$\seqsplit{abbaBAbACAccB}$,\ $\seqsplit{aCacBabbCAcACac}$\rangle$ &
$A$ &
\seqsplit{AAAAAAAAAcaCacBaBAbCaCacaB} \\[4pt]



$8_{21}$ & $\langle a,b \mid{}$\seqsplit{aaBBBaaBAbbbAAbaBBabAAbbbAB}$\rangle$ & $\mu={}$\seqsplit{Ab} & $\lambda={}$\seqsplit{AbAbAbaBBBaaBAbbAbbABaaBBBab} \\[4pt]

$9_{16}$ &
$\langle a,b \mid{}$\seqsplit{aabbaBabbaabABBAABBAbaabbaBabbaabABBAABabbaabbaBAABBAbABBAABabbaabbaBAABBAb}$\rangle$ &
\seqsplit{aabba} &
\seqsplit{BabbaabABBAABBAbaabbaBabbaabABBAABababababababababababababaBAABBAbaabbaBabbaabABBAABBAbaabbaBA} \\[4pt]


$9_{28}$ &
$\langle a,b \mid{}$\seqsplit{aabABBAABabbaBAABBAbaabbaBAABBBAABabbaabABBAABabbaBAABBAbaabbbaabABBAABabbaabABBAbaabbaBAABBAbaabbb}$\rangle$ &
\seqsplit{ba} & \seqsplit{babababABBAbaabbaBAABBBAABabbaabABBAABabbaBAABBAbaabbabbaabABBAABabbaBAABBAbaabbaBAABBBAABabbaabABBA} \\[4pt]

$9_{29}$ &
$\langle a,b,c \mid{}$\seqsplit{acbCbcABabcbC}$,\ $\seqsplit{acbCAbaBAbabABabcABabAB}$\rangle$ &
$B$ &
\seqsplit{BBBBBBCAbaBAbabCBacbCAbaBA} \\[4pt]

$9_{38}$ &
$\langle a,b,c \mid{}$\seqsplit{accaBabbC}$,\ $\seqsplit{abaBAbABAbCacBabaBabABAbCAcBabaBabAbCacB}$\rangle$ &
$A$ &
\seqsplit{cBabaBabABAbCAcBabCaBabaBAbABAbCacBabaBAAAAAAA} \\[4pt]

\end{tabular}
\caption{Some knot group presentations, computed with SnapPy \cite{SnapPy}, and the corresponding meridian and  null-homologous longitude. In this table
capitol letters indicate the inverse of the generator, meaning that $A= a^{-1}$.}
\label{table: presentation}
\end{table}

\begin{table}[h]
\centering
\small
\begin{tabular}{c|l}
$K$ & $\rho:\pi_1(S^3\setminus K)\to 2T$ \\
\hline
$8_5$   & $a\mapsto\tfrac12(1+i+j+k),\ b\mapsto-\tfrac12(1+i+j+k),\ c\mapsto\tfrac12(1-i-j+k)$ \\
$8_{15}$  & $a\mapsto\tfrac12(1+i+j+k),\ b\mapsto\tfrac12(1-i-j+k),\ c\mapsto i$ \\
$8_{21}$ & $a\mapsto\tfrac12(1+i+j+k),\ b\mapsto j$ \\
$9_{16}$  & $a\mapsto j,\ b\mapsto\tfrac12(-1+i+j+k)$ \\
$9_{28}$ & $a\mapsto\tfrac12(1+i+j+k),\ b\mapsto-i$ \\

\end{tabular}
\caption{Some irreducible representations killing the slope $(\mu)^{12}\lambda \subset \partial \compl{K}$.
The presentation of the groups are in Table \ref{table: presentation}.}
\label{tab:2T-12surgery-only}
\end{table}

\begin{prop}
\label{prop:8_5-8_15-9_16-never-abelian}
The knots $8_5,8_{15}, 8_{21}$, and $9_{28}$ are not $SU(2)$-abelian.
\begin{proof}
Let $K$ be one of the knots above. We are going to prove that for every $\nicefrac{p}{q} \in \mathbb{Q}$
the manifold $K(\nicefrac{p}{q})$ is not $SU(2)$-abelian.
By \cite{ApartialOrderInTheKnotTableII}, there is a surjective homomorphism $\varphi:\fund{\compl{K}}\twoheadrightarrow \fund{\compl{3_1}}$
with
\[
\varphi(\mu)=\mu', \qquad \varphi(\lambda)=(\lambda')^{\pm2},
\]
where $\mu,\lambda$ (resp.\ $\mu',\lambda'$) denote the meridian and null-homologous longitude of $K$ (resp. of $3_1$).

\textbf{Case 1: $12\nmid p$.} According to Remark \ref{rmk:su2-abundant-complement} the knot
$K$ is $SU(2)$-abundant.
By Definition \ref{defn: su2 abundant},
if $K(p/q)$ is $SU(2)$-abelian, then
there exists a $p$-th root of unity $\zeta$ with $\Delta_K(\zeta^2)=0$.
As shown in Table \ref{table: alexander polynomial}, the $\zeta^2$ must be a primitive $6$-th root of unity.
In particular, this means that $\mathrm{ord}(\zeta^2)=6$.
Since $\mathrm{ord}(\zeta^2)=\mathrm{ord}(\zeta)/\gcd(\mathrm{ord}(\zeta),2)$, this forces $\mathrm{ord}(\zeta)=12$.
This implies that $12\mid p$, which contradicts our assumption.
So $K(p/q)$ admits an irreducible representation.

\textbf{Case 2: $12\mid p$, say $p=12m$.}
Suppose for contradiction that $K(p/q)$ is $SU(2)$-abelian. The epimorphism $\varphi$
induces the following epimorphism
\[ 
    \fund{K(12m/q)} = \frac{\fund{\compl{K}}}{\normalsubgroup{\mu^{12m} \lambda^q}} \twoheadrightarrow \frac{\fund{\compl{3_1}}}{\normalsubgroup{(\mu')^{12m} (\lambda')^{\pm 2q}}}
    \twoheadrightarrow \frac{\fund{\compl{3_1}}}{\normalsubgroup{(\mu')^{6m} (\lambda')^{\pm q}}} = \fund{3_1(\pm 6m/q)}.
\]
We recall that $3_1=T(2,3)$.
According to Theorem \ref{thm:iterated-cables}, if $|\pm 6m/q| \neq 6$, then $\fund{3_1(\pm 6m/q)}$ admits an irreducible
$SU(2)$-representation. 
This implies that if $12m / q \neq 12$, then $\fund{K(12m/q)}$ admits an irreducible $SU(2)$-representation
by composition.
Table \ref{tab:2T-12surgery-only} provides an irreducible $SU(2)$-representation of $\fund{K(12)}$.
We conclude
that $K$ does not admit a non-trivial $SU(2)$-surgery.
\end{proof}
\end{prop}

\begin{prop}\label{prop: knots 9 29 and 9 38}
    The knots $9_{29}$ and $9_{38}$ are not $SU(2)$-abelian.
\begin{proof}
Let $K$ be one of the knots above.
As before, Table \ref{table: alexander polynomial} shows that $\Delta_K(t)$ has $\Phi_6(t)$ as its unique cyclotomic factor
and by Remark \ref{rmk:su2-abundant-complement} the knot $K$ is $SU(2)$-abundant.
By the same argument as in Proposition \ref{prop:8_5-8_15-9_16-never-abelian},
if $12\nmid p$, then $K(p/q)$ admits an irreducible $SU(2)$ representation.

It remains to treat the case $12\mid p$.
Table \ref{tab:929-938-progression} shows for both knots an explicit non-abelian representation $\rho:\pi_1(S^3\setminus K)\to 2T\subset SU(2)$ satisfying
\[
\rho(\lambda)=1 \qquad\text{and}\qquad \rho(\mu)^6=1.
\]
For such $\rho$ and any $p\equiv0\pmod6$, we have
\[
\rho(\mu)^p\rho(\lambda)^q = \rho(\mu)^p = \big(\rho(\mu)^6\big)^{p/6} = 1,
\]
so $\rho$ kills $\mu^p\lambda^q$ for every $q$,
and therefore descends to an irreducible representation of $\pi_1(K(p/q))\cong\fund{\compl{K}}/\normalsubgroup{\mu^p\lambda^q}$.
In particular this covers every $p$ with $12\mid p$.

Combining the two cases, $K(p/q)$ admits an irreducible $SU(2)$ representation for every $p/q\in\mathbb{Q}$.
As always, the conclusion is given by \cite[Theorem 1.1]{KronheimerMrowkaDehnSurgeryFundamental} where it is proven that
$K(0)$ is not $SU(2)$-abelian.
\end{proof}
\end{prop}

\begin{table}[h]
\centering
\small
\begin{tabular}{c|l}
$K$ & $\rho:\pi_1(S^3\setminus K)\to 2T$ \\
\hline
$9_{29}$&
$a\mapsto\tfrac12(-1+i+j+k),\ b\mapsto\tfrac12(1-i-j+k),\ c\mapsto-\tfrac12(1+i+j+k)$ \\[6pt]
$9_{38}$ &
$a\mapsto\tfrac12(1+i+j+k),\ b\mapsto-k,\ c\mapsto-i$ \\
\end{tabular}
\caption{Representations satisfying $\rho(\lambda)=1$, $\rho(\mu)^6=1$.
The presentation of the groups are in Table \ref{table: presentation}}
\label{tab:929-938-progression}
\end{table}

\begin{repcor}{cor: knots up to 9 crosssings}
    Let $K$ be a knot with at most $9$ crossings. If $K$ is neither a torus knot nor in $\{9_{47},9_{49}\}$,
    then
    $K$ is not $SU(2)$-abelian.
    \begin{proof}
        If $K$ is a $2$-bridge knot, then the conclusion is given by \cite[Theorem 1.1]{bowden2007winding}.
        The conclusion of the remaining knots is summarized in Table \ref{tab:su2-master}.
    \end{proof}
\end{repcor}

\begin{table}[h]
\centering
\small
\begin{minipage}{0.48\textwidth}
\centering
\begin{tabular}{c c l}
\toprule
\textbf{Knot} & \textbf{$SU(2)$-abelian?} & \textbf{Why?} \\
\midrule
$3_1$    & Yes & torus knot \\
$4_1$    & No  & 2-bridge \\
$5_1$    & Yes & torus knot\\
$5_2$    & No  & 2-bridge \\
$6_1$    & No  & 2-bridge \\
$6_2$    & No  & 2-bridge \\
$6_3$    & No  & 2-bridge \\
$7_1$    & Yes & torus knot\\
$7_2$    & No  & 2-bridge \\
$7_3$    & No  & 2-bridge \\
$7_4$    & No  & 2-bridge \\
$7_5$    & No  & 2-bridge \\
$7_6$    & No  & 2-bridge \\
$7_7$    & No  & 2-bridge \\
$8_1$    & No  & 2-bridge \\
$8_2$    & No  & 2-bridge \\
$8_3$    & No  & 2-bridge \\
$8_4$    & No  & 2-bridge \\
$8_5$    & No   & Proposition \ref{prop:8_5-8_15-9_16-never-abelian}          \\
$8_6$    & No  & 2-bridge \\
$8_7$    & No  & 2-bridge \\
$8_8$    & No  & 2-bridge \\
$8_9$    & No  & 2-bridge \\
$8_{10}$   &  No  &  Corollary \ref{corollary nodi}         \\
$8_{11}$   & No  & 2-bridge \\
$8_{12}$   & No  & 2-bridge \\
$8_{13}$   & No  & 2-bridge \\
$8_{14}$   & No  & 2-bridge \\
$8_{15}$   & No    & Proposition \ref{prop:8_5-8_15-9_16-never-abelian}         \\
$8_{16}$   & No   & Corollary \ref{cor:not-su2-abelian}         \\
$8_{17}$   &  No   & Corollary \ref{cor:not-su2-abelian}         \\
$8_{18}$   &  No   &  Corollary \ref{corollary nodi}         \\
$8_{19}$   & Yes & torus knot\\
$8_{20}$   &  No   & Corollary \ref{corollary nodi}         \\
$8_{21}$   &  No   & Proposition \ref{prop:8_5-8_15-9_16-never-abelian}         \\
$9_1$    & Yes & torus knot\\
$9_2$    & No  & 2-bridge \\
$9_3$    & No  & 2-bridge \\
$9_4$    & No  & 2-bridge \\
$9_5$    & No  & 2-bridge \\
$9_6$    & No  & 2-bridge \\
$9_7$    & No  & 2-bridge \\
\bottomrule
\end{tabular}
\end{minipage}
\hfill
\begin{minipage}{0.48\textwidth}
\centering
\begin{tabular}{c c l}
\toprule
\textbf{Knot} & \textbf{$SU(2)$-abelian?} & \textbf{Why?} \\
\midrule
$9_8$    & No  & 2-bridge \\
$9_9$    & No  & 2-bridge \\
$9_{10}$   & No  & 2-bridge \\
$9_{11}$   & No  & 2-bridge \\
$9_{12}$   & No  & 2-bridge \\
$9_{13}$   & No  & 2-bridge \\
$9_{14}$   & No  & 2-bridge \\
$9_{15}$   & No  & 2-bridge \\
$9_{16}$   & No  & Proposition \ref{prop:8_5-8_15-9_16-never-abelian}         \\
$9_{17}$   & No  & 2-bridge \\
$9_{18}$   & No  & 2-bridge \\
$9_{19}$   & No  & 2-bridge \\
$9_{20}$   & No  & 2-bridge \\
$9_{21}$   & No  & 2-bridge \\
$9_{22}$   & No   & Corollary \ref{cor:not-su2-abelian}         \\
$9_{23}$   & No  & 2-bridge \\
$9_{24}$   & No  & Corollary \ref{corollary nodi}         \\
$9_{25}$   &  No   & Corollary \ref{cor:not-su2-abelian}         \\
$9_{26}$   & No  & 2-bridge \\
$9_{27}$   & No  & 2-bridge \\
$9_{28}$   &  No   & Proposition \ref{prop:8_5-8_15-9_16-never-abelian}         \\
$9_{29}$   & No & Proposition \ref{prop: knots 9 29 and 9 38}       \\
$9_{30}$   &  No   & Corollary \ref{cor:not-su2-abelian}         \\
$9_{31}$   & No  & 2-bridge \\
$9_{32}$   & No    &   Corollary \ref{cor:not-su2-abelian}       \\
$9_{33}$   &  No   & Corollary \ref{cor:not-su2-abelian}         \\
$9_{34}$   &  No   &  Corollary \ref{cor:not-su2-abelian}         \\
$9_{35}$   &  No   &  Corollary \ref{cor:not-su2-abelian}         \\
$9_{36}$   &  No   &  Corollary \ref{cor:not-su2-abelian}        \\
$9_{37}$   &  No   &  Corollary \ref{corollary nodi}         \\
$9_{38}$   &  No   & Proposition \ref{prop: knots 9 29 and 9 38}         \\
$9_{39}$   &  No   & Corollary \ref{cor:not-su2-abelian}         \\
$9_{40}$   &  No   & Corollary \ref{corollary nodi}        \\
$9_{41}$   &  No   & Corollary \ref{cor:not-su2-abelian}         \\
$9_{42}$   &  No   & Corollary \ref{cor:not-su2-abelian}         \\
$9_{43}$   &  No   & Corollary \ref{cor:not-su2-abelian}         \\
$9_{44}$   &  No   & Corollary \ref{cor:not-su2-abelian}         \\
$9_{45}$   &  No   & Corollary \ref{cor:not-su2-abelian}         \\
$9_{46}$   &  No   & Corollary \ref{cor:not-su2-abelian}         \\
$9_{47}$   &  ???   &          \\
$9_{48}$   &  No   & Corollary \ref{cor:not-su2-abelian}         \\
$9_{49}$   &  ???   &          \\
\bottomrule
\end{tabular}
\end{minipage}
\caption{$SU(2)$-abelian status of prime knots with crossing number $\le 9$.}
\label{tab:su2-master}
\end{table}

\printbibliography
\end{document}